\documentclass[reqno]{amsart}

\usepackage [mathscr]{eucal}
\usepackage{amsmath,amsthm,amssymb,amscd}
\usepackage{amsrefs}
\usepackage{epsf}
\usepackage{amsfonts}
\usepackage{dsfont}
\usepackage{latexsym}
\usepackage{mathrsfs}
\usepackage{layout}
\usepackage{bm}
\usepackage{verbatim}

\newtheorem{theorem}{Theorem}[section]
\newtheorem{lemma}[theorem]{Lemma}
\newtheorem{proposition}[theorem]{Proposition}
\newtheorem{corollary}[theorem]{Corollary}

\newtheorem{definition}[theorem]{Definition}

\numberwithin{equation}{section}

\def\Xint#1{\mathchoice
{\XXint\displaystyle\textstyle{#1}}%
{\XXint\textstyle\scriptstyle{#1}}%
{\XXint\scriptstyle\scriptscriptstyle{#1}}%
{\XXint\scriptscriptstyle\scriptscriptstyle{#1}}%
\!\int}
\def\XXint#1#2#3{{\setbox0=\hbox{$#1{#2#3}{\int}$ }
\vcenter{\hbox{$#2#3$ }}\kern-.6\wd0}}

\def\dint{\Xint-}

\DeclareMathOperator *{\essosc}{ess\ osc}
\DeclareMathOperator *{\osc}{osc}
\DeclareMathOperator *{\esssup}{ess\ sup}
\DeclareMathOperator *{\essinf}{ess\ inf}
\DeclareMathOperator *{\di}{div} 
\DeclareMathOperator *{\meas}{meas}
\DeclareMathOperator *{\dist}{dist}
\DeclareMathOperator *{\data}{data}
\DeclareMathOperator *{\diam}{diam}
\DeclareMathOperator *{\loc}{loc}
\DeclareMathOperator *{\Lip}{Lip}
\DeclareMathOperator *{\Tr}{Tr}
\DeclareMathOperator *{\BMO}{BMO}
\DeclareMathOperator *{\Proj}{Proj}
\DeclareMathOperator *{\BBR}{\mathbb{R}}
\DeclareMathOperator *{\BBC}{\mathbb{C}}

\newcommand{\cJ}{{\mathcal J}}

\begin{document}
\title[Elliptic PDEs with complex coefficients]{
Regularity theory for solutions to second order elliptic operators with complex coefficients and the $L^p$ Dirichlet problem}

\author{Martin Dindo\v{s}}
\address{School of Mathematics, \\
         The University of Edinburgh and Maxwell Institute of Mathematical Sciences, UK}
\email{M.Dindos@ed.ac.uk}

\author{Jill Pipher}
\address{Department of Mathematics, \\ 
	Brown University, USA}
\email{jill\_pipher@brown.edu}

\begin{abstract}
We establish a new theory of regularity for elliptic complex valued second order equations of the form 
$\mathcal L=$div$A(\nabla\cdot)$, when the coefficients of the matrix $A$ satisfy a natural algebraic condition, a strengthened version of a condition known in the literature as $L^p$-dissipativity.  Precisely, the regularity result is a reverse H\"older condition for $L^p$ averages of solutions on interior balls, and serves as a replacement for the De Giorgi - Nash - Moser regularity of solutions to real-valued divergence form elliptic operators.
In a series of papers, Cialdea and Maz'ya studied necessary and sufficient conditions for $L^p$-dissipativity of second order complex coefficient operators and systems. Recently,  Carbonaro and Dragi\v{c}evi\'c  introduced a condition they termed $p$-ellipticity, and showed that it had implications for boundedness of certain bilinear operators that arise from complex valued second order differential operators. Their $p$-ellipticity condition is exactly our strengthened version of $L^p$-dissipativity. The regularity results of the present paper are applied to solve $L^p$ Dirichlet problems for  $\mathcal L=$div$A(\nabla\cdot)+B\cdot\nabla$ when $A$ and $B$ satisfy a Carleson measure condition, which previously was known only in the real valued case. We show solvability of the $L^2$ Dirichlet problem, as well as solvability of the $L^p$ Dirichlet boundary value problem for $p$ in the range where $A$ is $p$-elliptic. 
\end{abstract}

\maketitle

\section{Introduction}\label{S:Intro}

In this paper, we establish a new theory of interior regularity of solutions to second order divergence form complex coefficient operators 
$\mathcal{L}=\mbox{div} A(x)\nabla +B(x)\cdot\nabla$ under certain natural algebraic conditions on the matrix $A$ and a natural minimal scaling condition on $B$,
without any additional smoothness of the coefficients.
If the coefficients of $A$ and $B$ are real, the algebraic conditions on $A$ are precisely uniform ellipticity.

The improvements in regularity of solutions, as expressed as  \eqref{RHthm1} and \eqref{gradp} of Theorem \ref{Regularity} below, can be used as a substitute for the De Giorgi-Nash-Moser regularity theory for real divergence form elliptic equations. In the latter case, we know that when $A$, $B$ are real valued and $A$ is elliptic, the regularity theory for solutions gives that $u\in C^{\alpha}(B)$;  
this need not hold for solutions to complex coefficient operators. Indeed, solutions need not be locally bounded. 
However, we show that an iterative procedure, reminiscent of Moser's iteration scheme, can be amplified in a range determined by these algebraic conditions, yielding greater regularity of weak solutions. Moreover, we apply this regularity to show solvability of a Dirichlet boundary value problem for a class of variable coefficient complex coefficient second order operators, with certain minimal and natural smoothness assumptions on the coefficients. 
The solvability of these boundary value problems in the complex case require new ideas which make the methods of \cite{KP01} and \cite{DPP} more broadly applicable. Further,
the regularity theory established in this paper should make possible the study of boundary value problems for a variety of complex coefficient operators. 

We recently found that Carbonaro and Dragi\v{c}evi\'c \cite{CD} have also formulated this same algebraic condition (which they termed $p$-ellipticity) and showed that it had implications for boundedness of certain bilinear operators that arise
from complex valued second order differential operators. Their dimension-free bounds are related to the question of contractivity of the associated semigroup of these operators in $L^p$. The issue of $L^p$ contractivity of semigroups had been considered in a series of papers by Cialdea and Maz'ya \cite{CM2, CM1, CM3}, and we next introduce and discuss this in more detail.

In the case of complex coefficients, the usual ellipticity assumption is that 
there exist constants $0<\lambda\leq\Lambda<\infty$ such that
\begin{equation}\label{EllipAA}
\lambda|\xi|^{2}\leq \mathscr{R}e\,\sum_{i,j=0}^{n-1} 
 A_{ij}(x)\xi_{i}\overline{\xi_{j}}=\mathscr{R}e\,\langle A(x)\xi,\xi\rangle\quad\mbox{and}\quad |\langle A\xi,\eta\rangle|\le \Lambda|\xi||\eta|
\end{equation}
for all $\xi,\,\eta\in \BBC^n$ and a.e. $x\in\Omega$. 
In this paper, we consider a stronger form of ellipticity, a strengthening of the concept of 
$L^p$ dissipativity as defined in \cite{CM2, CM1, CM3}, which in turn was motivated by understanding when semigroups generated by second order elliptic operators are 
contractive in $L^p$. In particular, it had long been known that scalar second order elliptic operators with real coefficients generate contractive semigroups in $L^p$ for all $1\le p\le\infty$.

In \cite{CM2}, the following condition was shown to be sufficient for $L^p$ dissipativity:
\begin{equation}
\frac4{pp'}\langle\mathscr{R}e\,A(x)\xi,\xi\rangle+\langle(\mathscr{R}e\,A(x)\eta,\eta\rangle+2\langle p^{-1}
\mathscr{I}m\,A(x)-{p'}^{-1}\mathscr{I}m\,A^t(x))\xi,\eta\rangle\ge 0,
\end{equation}
for all $\xi,\,\eta\in\mathbb R^n$.

We will consider a stronger condition, and use a
change of variables $\xi=\frac{\sqrt{pp'}}{2}\lambda$ to write it as follows: For some $\varepsilon>0$ and all 
$\lambda,\eta\in\mathbb R^n$
\begin{equation}\label{CMell}
\langle \mathscr{R}e\,A\,\lambda,\lambda\rangle+\langle \mathscr{R}e\,A\,\eta,\eta\rangle+
\left\langle \left(\textstyle\sqrt{\frac{p'}{p}}
\mathscr{I}m\,A-\sqrt{\frac{p}{p'}}\mathscr{I}m\,A^t\right)\lambda,\eta\right\rangle
\ge \varepsilon(|\lambda|^2+|\eta|^2).
\end{equation}


This same condition also appears in the paper \cite{CD}, where the authors introduce it in the following form. For $p>1$ define the ${\mathbb R}$-linear map $\cJ_p:{\mathbb C}^n\to {\mathbb C}^n$ by
$$\cJ_p(\alpha+i\beta)=\frac{\alpha}{p}+i\frac{\beta}{p'}$$
where $p'=p/(p-1)$ and $\alpha,\beta\in{\mathbb R}^n$. They define the matrix $A$ to be $p$-elliptic if
for a.e. $x\in\Omega$
\begin{equation}\label{pEll}
\mathscr{R}e\,\langle A(x)\xi,\cJ_p\xi\rangle \ge \lambda_p|\xi|^2,\qquad\forall \xi\in{\mathbb C}^n
\end{equation}
for some $\lambda_p>0$.  

Henceforth, $A$ {\it will be called $p$-elliptic} if it satisfies \eqref{pEll} and the upper bound
\begin{equation}
|\langle A(x)\xi,\eta \rangle| \le \Lambda |\xi| |\eta|, \qquad\forall \xi, \,\eta\in{\mathbb C}^n.
\end{equation}
A short calculation shows that \eqref{pEll} and \eqref{CMell} are equivalent.

We shall adopt the notation of \cite{CD},  and recall their observation that
this strengthened ellipticity condition is equivalent to $\Delta_p(A)>0$ where
\begin{equation}\label{Deltap}
\Delta_p(A)=\essinf_{x\in\Omega}\min_{|\xi|=1}\mathscr{R}e\,\langle A(x)\xi,\cJ_p\xi\rangle. 
\end{equation}
Observe that when $p=2$ this is just the usual ellipticity condition \eqref{EllipAA}. The $p$-ellipticity condition $\Delta_p(A)>0$ can be restated in a different form
\begin{equation}\label{Ellmu}
|1-2/p|<\mu(A),
\end{equation}
where
\begin{equation}\label{Defmu}
\mu(A)=\essinf_{(x,\xi)\in\Omega\times{\mathbb C}^n\setminus\{0\}}\mathscr{R}e\,\frac{\langle A(x)\xi,\xi\rangle}{|\langle A(x)\xi,\overline{\xi}\rangle|}.
\end{equation}
(c.f. Proposition 5.14 of \cite{CD}). The advantage of writing the inequality in this form is that it separates $A$ from $p$. It also immediately implies that if a matrix $A$ is elliptic (i.e., \eqref{EllipAA} holds) then there exists $p_0\in [1,2)$
such that $A$ is $p$-elliptic if and only if $p\in (p_0,p_0')$, where $p_0=2/(1+\mu(A))$. Moreover, $p_0=1$ if and only if the matrix $A$ is real and the quantity $\mu(A)$ is trivially bounded by $\mu(A)\ge\lambda/\Lambda$ giving a trivial upper bound on the value of $p_0$.\vglue2mm

Our first main result concerns solutions to $\mathcal{L}=\mbox{div} A(x)\nabla +B(x)\cdot\nabla$  in a domain $\Omega\subset {\mathbb R}^n$.

\begin{theorem}\label{Regularity}
Suppose that  $u\in W^{1,2}_{loc}(\Omega;{\BBC})$ is the weak solution to the operator ${\mathcal L}u:=\mbox{\rm div} A(x)\nabla u +B(x)\cdot\nabla u = 0$ in $\Omega$.
Let $p_0 = \inf \{p>1: \text{$A$ is $p$-elliptic}\}$,
and suppose that $B$ has measurable coefficients  $B_i \in L^{\infty}_{loc}(\Omega)$ satisfying the condition
\begin{equation}\label{Bcond}
|B_i(x)| \leq K (\delta(x))^{-1}, \quad\forall x \in \Omega
\end{equation}
where the constant $K$ is uniform, and $\delta(x)$ denotes the distance of $x$ to the boundary of $\Omega$. 
Then we have the following improvement in the regularity of $u$. For any $B_{4r}(x)\subset\Omega$ and $\varepsilon>0$ there exists $C_\varepsilon>0$ such that
\begin{equation}\label{RHthm1}
\left(\dint_{B_{r}(x)} |u|^{p} \,dy\right)^{1/{p}}
\le
C_\varepsilon\left(\dint_{B_{2 r}(x)} |u|^{q} \,dy\right)^{1/{q}}+\varepsilon \left(\dint_{B_{2 r}(x)} |u|^{2} \,dy\right)^{1/{2}}
\end{equation}
for all $p,q \in (p_0, \frac{p'_0n}{n-2})$. (Here $p_0'=p_0/(p_0-1)$ and when $n=2$ one can take $p,q\in (p_0,\infty)$.) The constant in the estimate depends on the dimension, the $p$-ellipticity constants, $\Lambda$, $K$ and $\varepsilon>0$ but not on $x\in\Omega$, $r>0$ or $u$. 
Moreover, for all $p\in (p_0,p_0')$ and any $\varepsilon>0$
\begin{equation}\label{gradp}
r^2 \dint_{B_r(x)} |\nabla u(y)|^2 |u(y)|^{p-2} dy \le  C_\varepsilon \dint_{B_{2r}(x)} |u(y)|^p dy +\varepsilon \left(\dint_{B_{2r}(x)} |u(y)|^2 dy\right)^{p/2},
\end{equation}
where the constant again depend only on the dimension,  $p$, $\Lambda$, $K$ and $\varepsilon>0$.
In particular, $|u|^{(p-2)/2} u$ belongs to $W^{1,2}_{loc}(\Omega;\BBC)$.
\end{theorem}

\noindent{\it Remark.} Clearly, if $q\ge 2$ in \eqref{RHthm1} and if $p\ge 2$ in \eqref{gradp} one can take $\varepsilon=0$ as the $L^2$ average of $u$ can be controlled by the first term on the right hand side of each of the two inequalities.\vglue2mm

In general, one can not expect a larger range of $p$ in the reverse H\"older condition of \eqref{RHthm1}. In \cite{May}, Mayboroda gives a counterexample to \eqref{RHthm1} when $q=2$  and for any $p > \frac{2n}{n-2}$ under the assumption of \eqref{EllipAA} (which is the same as $2$-ellipticity). 

\smallskip

We apply this regularity result to the question of solvability of $L^p$ Dirichlet problem for elliptic operators of this type. This part of the paper is motivated by the known results concerning boundary value problems 
for second order elliptic equations in divergence form, when the coefficients  are {\bf real} and satisfy
a certain natural, minimal smoothness condition (refer \cite{DPP,DPR,KP01}).  The literature on solvability of boundary value problems for complex coefficient operators in $\mathbb R^n$ is limited, except when the matrix $A$ is of block form. For block matrices $A$ in $\mathcal L = \mbox{div} A(x)\nabla$, there are numerous results on on $L^p$-solvability of the Dirichlet, regularity and Neumann problems, starting with the solution of the Kato problem, where the coefficients of the block matrix are also assumed to be independent of the transverse variable (this assumption is usually referred in literature as \lq\lq $t$-independent", in our notation it is the $x_0$ variable).
See \cite{AHLMT} and \cite{HM} and the references therein. 
For matrices not of block form, there are solvability results in various special cases assuming that the solutions satisfy De Giorgi - Nash - Moser estimates. See \cite{AAAHK} and \cite{HKMPreg} for example. The latter paper is also concerned with operators that are $t$-independent.
Finally, there are perturbation results in a variety of special cases, such as \cite{AAM} and \cite{AAH}; the first paper shows that solvability in $L^2$ implies solvability in $L^p$ for $p$ near $2$, and the second paper has $L^2$-solvability results for small $L^\infty$ perturbations of real elliptic operators when the complex matrix is $t$-independent.

Our solvability result for operators of the form $\mathcal{L}=\mbox{div} A(x)\nabla+B(x)\cdot\nabla$ can be applied on a domains above a Lipschitz graph in $\mathbb R^n$. We do not assume \lq\lq $t$-independence".
Instead, we assume the coefficients 
$A$ and $B$ satisfy a natural Carleson condition that has appeared in the literature so far only for real elliptic operators.(\cite {KP01}, \cite{DPP}, and \cite{DPR}). The Carleson condition on $A$, \eqref{Car_hatAA} below, holds uniformly on Lipschitz subdomains, and is thus a natural condition in the context of chord-arc domains as well. The paper \cite{HMTo} connects geometric information about the boundary of the domain to information about the elliptic measure of operators that satisfy some closely related conditions. 

\smallskip

The second main theorem of the paper establishes the
solvability of $\mathcal Lu=0$ with $L^{p}$ Dirichlet boundary data for variable coefficient complex coefficient operators satisfying these Carleson conditions on coefficients. The solvability of the $L^2$ Dirichlet problem (where 2-ellipticity is the standard assumption) for complex coefficient operators satisfying  \eqref{Car_hatAA} is a consequence of the $L^2$ results for non-symmetric elliptic systems in \cite{DHM}. In fact, as is typical in this theory, Dindo\v{s}, Hwang, and M. Mitrea actually obtain in \cite{DHM} $L^p$ results for elliptic systems, for $p \in (2-\varepsilon, 2+\varepsilon)$.
One of the novelties of this paper is showing that $p$-ellipticity extends solvability of the $L^p$-Dirichlet problem to a broader range of $L^p$. To prove this, we adapt the method of proof that was established in \cite {KP01}; however, the lack of continuity of solutions and the absence of a maximum principle requires new ideas to generalize this approach. It turns out the $p$-ellipticity condition is also the correct one for establishing solvability of perturbations of elliptic operators as well as solvability of the regularity problems. We shall take this up in a separate manuscript.

\begin{theorem}\label{S3:T1}  Let $1<p<\infty$, and let $\Omega$ be the upper half-space ${\mathbb R}^n_+=\{(x_0,x'):\,x_0>0\mbox{ and } x'\in{\mathbb R}^{n-1}\}$. Consider the operator 
$$ \mathcal Lu = \partial_{i}\left(A_{ij}(x)\partial_{j}u\right) 
+B_{i}(x)\partial_{i}u$$
and assume that the matrix $A$ is $p$-elliptic with constants $\lambda_p,\Lambda$, $A_{00}=1$ and $\mathscr{I}m\,A_{0j}=0$ for all $1\leq j \leq n-1$.
Assume that
\begin{equation}\label{Car_hatAA}
d{\mu}(x)=\sup_{B_{\delta(x)/2}(x)}\left[|\nabla{A}(x)|^{2} + |{B}(x)|^{2} \right]\delta(x)\,dx
\end{equation}
is a Carleson measure in $\Omega$. Let us also denote 
\begin{equation}\label{Car_hatAAB}
d{\mu'}(x)=\sup_{B_{\delta(x)/2}(x)}\left[\textstyle\sum_j\left|\partial_0 A_{0j}\right|^{2}+\left|\textstyle\sum_j\partial_j A_{0j}\right|^{2} + |{B}(x)|^{2} \right]\delta(x)\,dx.
\end{equation}

Then there exist $K=K(\lambda_p,\Lambda,\|\mu\|_{\mathcal C},n,p)>0$ and $C(\lambda_p, \Lambda ,\|\mu\|_{\mathcal C}, n,p)>0$ such that if
\begin{equation}\label{Small-Cond}
\|\mu'\|_{\mathcal C} < K
\end{equation}
then the $L^p$-Dirichlet problem  
 
 \begin{equation}\label{E:D2}
\begin{cases}
\,\,{\mathcal L}u=0 
& \text{in } \Omega,
\\[4pt]
\quad u=f & \text{ for $\sigma$-a.e. }\,x\in\partial\Omega, 
\\[4pt]
\tilde{N}_{p,a}(u) \in L^{p}(\partial \Omega), &
\end{cases}
\end{equation}
is solvable and the estimate
\begin{equation}\label{Main-Est}
\|\tilde{N}_{p,a} (u)\|_{L^{p}(\partial \Omega)}\leq C\|f\|_{L^{p}(\partial \Omega;{\BBC})}
\end{equation}
holds for all energy solutions $u$ with datum $f$.
\end{theorem}

In the statement of this theorem, we've used some notation that will be defined in subsequent sections. We will also recall there the concept of Carleson measure, discuss the notions of $L^p$ solvability and energy solutions and define $\tilde{N}_{p}$ which is a variant of the nontangential maximal function defined using $L^p$ averages of the solution $u$. By Theorem \ref{Regularity}, instead of $L^p$ averages we could use $L^q$ averages for $q$ in the range $(p_0,\frac{p_0'n}{n-2})$ to obtain the same result. Also, see section \ref{S3} for a detailed discussion of some further assumptions we make to prove Theorem \ref{S3:T1} on Lipschitz domains.

Classically, the $L^p$ boundedness of the nontangential maximal function of a solution (in our case, estimate \eqref{Main-Est} above) gives nontangential convergence of the solution to its boundary values. Since the nontangential maximal function of our complex-valued solution will require smoothing by averaging, we will also get a nontangential convergence result, but stated for averages of solutions. This convergence of averages is a consequence of solvability in $L^p$ and is not connected with the assumptions on the coefficients of the equation. For this reason, that very general result is given in an appendix at the end.

We now state, as a further corollary of the second main theorem, a result for matrices $A$ in ``block form". This corollary uses the fact that the assumption that the Carleson measure norm is small is needed only on the last row ($A_{0j}$) of the matrix $A$. This latter observation was pointed out to us by S. Mayboroda - see \cite{DFM2} and also \cite{DFM}. Hence, we have the following.

\begin{corollary}\label{block} Suppose the operator $\mathcal L$ on $\mathbb R^n_+$ has the form
$${\mathcal L}u=\partial^2_0 u +\sum_{i,j=1}^{n-1}\partial_i(A_{ij}\partial_j u)$$
where the matrix $A$ has coefficients satisfying the Carleson condition \eqref{Car_hatAA}.

Then for all $1<p<\infty$ for which $A$ is $p$-elliptic, the $L^p$-Dirichlet problem  \eqref{E:D} is solvable for $\mathcal L$
and the estimate
\begin{equation}\label{Main-Est2X}
\|\tilde{N}_{p,a} u\|_{L^{p}(\partial \Omega)}\leq C\|f\|_{L^{p}(\partial \Omega;{\BBC})}
\end{equation}
holds for all energy solutions $u$ with datum $f$. 
\end{corollary}

\smallskip

{\it Acknowlegements.} We would like to thank S. Mayboroda for helpful comments on an earlier draft of this manuscript.


\section{$L^p$ dissipativity, $p$-ellipticity and regularity results}
\label{SS:Nor}

The concept of $L^p$ dissipativity was defined in a series of papers by Cialdea and Maz'ya \cite{CM2, CM1, CM3} and was motivated by the effort to characterize when semigroups generated by second order elliptic operators are 
contractive in $L^p$. In particular, it has long been known that scalar second order elliptic operators with real coefficients generate contractive semigroups in $L^p$ for all $1\le p\le\infty$. The case 
of operators generating $L^\infty$-contractive semigroups is studied in \cite{ABBO}.

Following \cite{CM2} let $\mathcal L(u,v)$ be the the sesquilinear form
$$\mathcal L(u,v)=\int_{\Omega} \langle A(x)\nabla u,\nabla v\rangle\, dx,$$
where $\langle\cdot,\cdot\rangle$ is the usual inner product on $\mathbb C^n$. Clearly,
$\mathcal L(u,v)$ is well defined for $(u,v)\in C^1_0(\Omega)\times C^1_0(\Omega)$ consisting of
complex value functions having compact support in $\Omega$ with continuous first derivative.

\begin{definition} (Cialdea-Maz'ya) Let $1<p<\infty$. The form $\mathcal L$ is called $L^p$ dissipative if for all
$u\in C_0^1(\Omega)$
\begin{eqnarray}
&&\mbox{Re }\mathcal L(u,|u|^{p-2}u)\ge 0,\qquad\mbox{if } p\ge 2,\\
&&\mbox{Re }\mathcal L(u|u|^{p'-2},u)\ge 0,\qquad\mbox{if } 1<p< 2.
\end{eqnarray}
\end{definition}

\begin{theorem} \label{TCM}\cite[Theorem 1, Corollary 4 and Corollary 6]{CM2} A sufficient condition for the form $\mathcal L$ to be $L^p$ dissipative is that
\begin{equation}
\frac4{pp'}\langle\mathscr{R}e\,A(x)\xi,\xi\rangle+\langle(\mathscr{R}e\,A(x)\eta,\eta\rangle+2\langle p^{-1}
\mathscr{I}m\,A(x)-{p'}^{-1}\mathscr{I}m\,A^t(x))\xi,\eta\rangle\ge 0,\label{Dcond}
\end{equation}
for all $\xi,\,\eta\in\mathbb R^n$. If in addition the matrix $\mathscr{I}m\,A$ be symmetric, i.e., $\mathscr{I}m\,A\,=\mathscr{I}m\,A^t$ then this condition is also necessary and is equivalent to
\begin{equation}\label{Acond}
|p-2||\langle \mathscr{I}m\,A(x)\xi,\xi\rangle|\le 2\sqrt{p-1}\langle\mathscr{R}e\,A(x)\xi,\xi\rangle,\qquad\forall \xi\in{\BBR}^n\mbox{ and a.e. }x\in\Omega.
\end{equation}
\eqref{Acond} must hold even in the non-symmetric case, but then this condition might not be sufficient.\vglue1mm

In particular, if we set 
\begin{equation}
\tilde\mu=\essinf_{(x,\xi) \in \mathcal M} 
\frac{\langle\mathscr{R}e\,A(x)\xi,\xi\rangle}{|\langle \mathscr{I}m\,A(x)\xi,\xi\rangle|}\label{MU}
\end{equation}
where 
$\mathcal M$ is the set of $(\xi,x)$, $x \in\mathbb R^n$, $x\in\Omega$ such that $\langle \mathscr{I}m\,A(x)\xi,\xi\rangle\ne 0$. If $\mathscr{I}m\,A=0$ for any $x\in\Omega$ then $\mathcal L$ is $L^p$ dissipative for all $p>1$. If $\mathscr{I}m\,A$ is symmetric but does not vanish identically on $\Omega$ then $\mathcal L$ is $L^p$ dissipative 
if and only if
\begin{equation}\label{Pcond}
2+2\tilde\mu\left(\tilde\mu-\sqrt{\tilde\mu^2+1}\right)\le p\le 2+2\tilde\mu\left(\tilde\mu+\sqrt{\tilde\mu^2+1}\right).
\end{equation}
\end{theorem}\vglue2mm

For our purposes \eqref{Dcond} is not sufficient as in particular this condition does not imply ellipticity when $p=2$. To guarantee ellipticity, a stronger lower bound is needed, namely that the left hand side of \eqref{Dcond} is greater than $\varepsilon(|\xi|^2+|\eta|^2)$. This yields the condition \eqref{CMell}, which when $p=2$ is just the usual ellipticity for complex coefficients.

As we have observed in the introduction, \eqref{CMell} can be simply written as $\Delta_p(A)>0$, where \eqref{Deltap} defines $\Delta_p(A)$. (This was introduced in \cite{CD}) and is in turn equivalent to $|1-2/p|<\mu(A)$). Hence the following holds.

\begin{theorem} \label{NDC} Let $A\in L^\infty(\Omega)$ be a matrix that is uniformly elliptic: 
for some $\lambda,\Lambda>0$ and almost every $x\in\Omega$ we have
\begin{equation}\label{EllipAAAA}
\lambda|\xi|^{2}\leq \mathscr{R}e\,\sum_{i,j=0}^{n-1} 
 A_{ij}(x)\xi_{i}\overline{\xi_{j}} \quad\text{and}\quad |\langle A\xi,\eta\rangle|\leq\Lambda|\xi||\eta| 
 \end{equation}
for all $\xi,\eta\in \BBC^n$. Then there exists $p_0\in [1,2)$ (with $p_0=1$ if and only if $\mathscr{I}m\,A\,=0$) such that the matrix $A$ is $p$-elliptic if and only if $p\in (p_0,p_0')$.  That is $\Delta_p(A)>0$, equivalently 
\begin{equation}
\langle \mathscr{R}e\,A\,\lambda,\lambda\rangle+\langle \mathscr{R}e\,A\,\eta,\eta\rangle+
\left\langle \left(\textstyle\sqrt{\frac{p'}{p}}
\mathscr{I}m\,A-\sqrt{\frac{p}{p'}}\mathscr{I}m\,A^t\right)\lambda,\eta\right\rangle
\ge \varepsilon(p)(|\lambda|^2+|\eta|^2),\label{DDDcond}
\end{equation}
for some $\varepsilon(p)>0$ and all $\lambda,\eta\in\mathbb R^n$. Here $p_0=\frac{2}{1+\mu(A)}$ where
\begin{equation}\label{Defmu2}
\mu(A)=\essinf_{(x,\xi)\in\mathcal M}\mathscr{R}e\,\frac{\langle A(x),\xi,\xi\rangle}{|\langle A(x),\xi,\overline{\xi}\rangle|}\ge\frac{\lambda}{\Lambda},
\end{equation}
where $\mathcal M$ is the set $(\xi,x)$, $x \in\mathbb C^n$, $x\in\Omega$ such that $\langle A(x)\xi,\overline{\xi}\rangle\ne 0$. 
If in addition the matrix $A$ has symmetric imaginary part ($\mathscr{I}m\,A\,=\mathscr{I}m\,A^t$) this further simplifies to
\begin{equation}\label{PPcond}
p_0=2+2\tilde\mu\left(\mu-\sqrt{\tilde\mu^2+1}\right),
\end{equation}
where 
\begin{equation}
\tilde\mu=\essinf_{(x,\xi)\in\widetilde{\mathcal M}}\frac{\langle\mathscr{R}e\,A(x)\xi,\xi\rangle}{|\langle \mathscr{I}m\,A(x)\xi,\xi\rangle|}\label{MUU}
\end{equation}
and $\widetilde{\mathcal M}$ is the set of $(\xi,x)$, $\xi \in\mathbb R^n$, $x\in\Omega$ such that $\langle \mathscr{I}m\,A(x)\xi,\xi\rangle\ne 0$. 
\end{theorem}

We apply the concept of $p$-ellipticity to prove the following result.

\begin{theorem}\label{MTdiss} Assume that the matrix $A$ is $p$-elliptic. Then there exists $\lambda'_p=\lambda'_p(\Lambda,\lambda_p)>0$ such that for any nonnegative, bounded and measurable function $\chi$ and any
 $u$ such that $|u|^{(p-2)/2} u \in W^{1,2}_{loc}(\Omega;\BBC)$, we have
\begin{equation}\label{SD}
\mathscr{R}e\,\int_{\Omega}\langle A(x)\nabla u,\nabla(|u|^{p-2}u)\rangle\chi(x)\, dx\ge \lambda'_p \int_{\Omega}|u|^{p-2}|\nabla u|^2\chi(x)\,dx.
\end{equation}
\end{theorem}

\begin{proof} Since $A$ is $p$-elliptic \eqref{DDDcond} holds. Changing the variables $\lambda=\frac2{\sqrt{pp'}}\xi$
and obtain
\begin{align}
\frac4{pp'}\langle \mathscr{R}e\,A\,\xi,\xi\rangle&+\langle \mathscr{R}e\,A\,\eta,\eta\rangle+
\nonumber\\&+2\langle p^{-1}
\mathscr{I}m\,A-{p'}^{-1}\mathscr{I}m\,A^t)\xi,\eta\rangle\ge \varepsilon'(|\xi|^2+|\eta|^2).\label{eq2.27}
\end{align}

Consider now $v\in W^{1,2}_{loc}(\Omega;\BBC)$ and write $|v|^{-1}\overline{v}\nabla v$ as $X+iY$, that is 
$$X=\mathscr{R}e(|v|^{-1}\overline{v}\nabla v)\quad\mbox{and}\quad Y=\mathscr{I}m(|v|^{-1}\overline{v}\nabla v).$$

Let $\chi(x)$ be a nonnegative bounded an measurable function on $\Omega$.
By using $X$ in place of $\xi$ and $Y$ in place of $\eta$, multiplying by $\chi(x)$ and integrating over $\Omega$ one obtains from \eqref{eq2.27} (similar to \cite[Corollary 4]{CM2})
\begin{eqnarray}
&&\mathscr{R}e\, \int_{\Omega}\Big[\langle A\nabla v,\nabla v\rangle-(1-2/p)\langle(A-A^*)\nabla(|v|),|v|^{-1}\overline{v}\nabla v\rangle\nonumber\\
&&\qquad-(1-2/p)^2\langle A\nabla(|v|),\nabla(|v|) \rangle\Big]\chi(x)\, dx\ge \varepsilon'\int_{\Omega}|\nabla v|^2\chi(x)\,dx
\end{eqnarray}
Now as in (2.9) of \cite{CM2} one considers $v=|u|^{p/2-1}u$ and
$$g_{\varepsilon}=(|v|^2+\varepsilon^2)^{1/2},\qquad u_{\varepsilon}=g_{\varepsilon}^{2/p-1}v,$$
for $u,\,u_\varepsilon$. Using our assumption we see that $|u|^{(p-2)/2}u,\, |u_\varepsilon|^{(p-2)/2}u_\varepsilon\in W^{1,2}_{loc}(\Omega,{\mathbb C}^n)$. We let $\varepsilon\to 0+$. When $p>2$, we obtain, using Lebesgue dominated convergence,
\begin{eqnarray}\label{Sdissip}&&\lim_{\varepsilon\to 0+}\,\mathscr{R}e\,\int_{\Omega}\langle A\nabla u_{\varepsilon},\nabla(|u_\varepsilon|^{p-2}u_\varepsilon)\rangle\chi(x)\, dx=\\
&&\mathscr{R}e\, \int_{\Omega}\Big[\langle A\nabla v,\nabla v\rangle-(1-2/p)\langle(A-A^*)\nabla(|v|),|v|^{-1}\overline{v}\nabla v\rangle\nonumber\\
&&\quad-(1-2/p)^2\langle A\nabla(|v|),\nabla(|v|) \rangle\Big]\chi(x)\, dx\ge \varepsilon'\int_{\Omega}|\nabla (|u|^{p/2-1}u)|^2\chi(x)\,dx.\nonumber
\end{eqnarray}
and then the first term of \eqref{Sdissip} is just 
\begin{equation}
\mathscr{R}e\,\int_{\Omega}\langle A(x)\nabla u,\nabla(|u|^{p-2}u)\rangle\chi(x)\, dx.\label{eq2.31}
\end{equation}

When $1<p<2$, we use a duality argument based on an observation in \cite{CM2}. Set $w=|u|^{p-2}u$, so that $u = |w|^{p'-2}w$. The fact that $|u|^{(p-2)/2} u \in W^{1,2}_{loc}(\Omega;\BBC)$ implies that
$|w|^{(p'-2)/2} w \in W^{1,2}_{loc}(\Omega;\BBC)$. 
Then,
$$\mathscr{R}e\,\int_{\Omega}\langle A(x)\nabla u,\nabla(|u|^{p-2}u)\rangle\chi(x)\, dx =
\mathscr{R}e\,\int_{\Omega}\langle A^*(x)\nabla w,\nabla(|w|^{p'-2}w)\rangle\chi(x)\, dx,
$$
and we have that $A^*$ is $p'$-elliptic when $A$ is $p$-elliptic.
Therefore, we have reduced the regime of the first case considered, and the same limiting argument yields  
\begin{equation}
\mathscr{R}e\,\int_{\Omega}\langle A^*(x)\nabla w,\nabla(|w|^{p'-2}w)\rangle\chi(x)\, dx\ge \lambda'_{p'} \int_{\Omega}|w|^{p'-2}|\nabla w|^2\chi(x)\,dx.
\end{equation}

We conclude the argument for \eqref{SD} in the case $p<2$ by observing that
$$\int_{\Omega}|w|^{p'-2}|\nabla w|^2\chi(x)\,dx \approx \int_{\Omega}|\nabla (|w|^{p'/2-1}w)|^2 \chi(x)\,dx = \int_{\Omega}|\nabla (|u|^{p/2-1}u)|^2 \chi(x)\,dx.$$
\end{proof}

The following observation will be used frequently.

\begin{lemma}\label{l2.8} For all $p>1$, and for all $x$ for which  $u(x) \ne 0$
$$|\nabla (|u(x)|^{p/2-1}u(x))|^2\approx |u(x)|^{p-2}|\nabla u(x)|^2.$$
\end{lemma}

\begin{proof} For $k=0,1,\dots,n-1$ we have
$$\partial_k(|u|^{(p-2)/2}u)=|u|^{(p-4)/2}\left[|u|\partial_ku+\textstyle\frac{p-2}2u\partial_k(|u|)\right].$$
Multiplying by its complex conjugate then yields
$$\left|\partial_k(|u|^{(p-2)/2}u)\right|^2=|u|^{p-4}\Big[|u|^2|\partial_k u|^2+(p-2)|u|\partial_k(|u|)\mathscr{R}e\,\langle u,\partial_k u\rangle$$
$$+\left(\textstyle\frac{p-2}{2}\right)^2|u|^2(\partial_k(|u|))^2\Big].$$
As in the proof of \cite[Lemma 2]{CM1} we have  $\partial_k(|u|)=|u|^{-1}\mathscr{R}e\,\langle u,\partial_k u\rangle$ which yieds
$$\left|\partial_k(|u|^{(p-2)/2}u)\right|^2=|u|^{p-4}\Big[|u|^2|\partial_k u|^2+\left[(p-2)+\left(\textstyle\frac{p-2}{2}\right)^2\right]\left|\mathscr{R}e\,\langle u,\partial_k u\rangle\right|^2\Big].$$
Summing over all $k$ gives us
$$\left|\nabla(|u|^{(p-2)/2}u)\right|^2=|u|^{p-4}\Big[|u|^2|\nabla u|^2+\left[\left(\textstyle\frac{p}{2}\right)^2-1\right]\displaystyle\sum_k\left|\mathscr{R}e\,\langle u,\partial_k u\rangle\right|^2\Big].$$
Since by Cauchy-Schwarz
$$\sum_k\left|\mathscr{R}e\,\langle u,\partial_k u\rangle\right|^2\le |u|^2|\nabla u|^2,$$
then clearly
$$\left|\nabla(|u|^{(p-2)/2}u)\right|^2\le \left(1+\left|\left(\textstyle\frac{p}{2}\right)^2-1\right|\right)|u|^{p-2}|\nabla u|^2,$$
and for $p>0$
$$\min\left\{1,\left(\textstyle\frac{p}2\right)^2\right\}|u|^{p-2}|\nabla u|^2\le \left|\nabla(|u|^{(p-2)/2}u)\right|^2.$$
Hence the claim holds.
\end{proof}

We now turn to the main lemmas required to prove Theorem \ref{Regularity}.

\begin{lemma}\label{LpAvebig}
Let the matrix $A$ be $p$-elliptic for $p \geq 2$ and let $B$ have coefficients satisfying condition \eqref{Bcond} of Theorem \ref{Regularity}.
Suppose that $u$ is a $W^{1,2}_{loc}(\Omega;\BBC)$ solution to $\mathcal L$ in $\Omega$. Then, for any ball $B_r(x)$ with $r < \delta(x)/4$,
\begin{equation}\label{psq}
\int_{B_r(x)} |\nabla u(y)|^2 |u(y)|^{p-2} dy \lesssim r^{-2} \int_{B_{2r}(x))} |u(y)|^p dy
\end{equation}
and
\begin{equation}\label{pAve}
\left(\dint_{B_{r}(x))} |u(y)|^q dy\right)^{1/q}
\lesssim  \left(\dint_{B_{2r}(x)} |u(y)|^2 dy\right)^{1/2}
\end{equation}
for all $q \in (2,\frac{np}{n-2}]$ when $n > 2$, and
where the implied constants depend only $p$-ellipticity and $K$ of \eqref{Bcond}.
When $n=2$, $q$ can be any number in $(2,\infty)$.
In particular, $|u|^{(p-2)/2} u$ belongs to $W^{1,2}_{loc}(\Omega;\BBC).$
\end{lemma}

\begin{proof}   

We begin by assuming that both $A$ and $B$ are smooth. Thus the solution $u$ to ${\mathcal L}(u) = 0$ will be smooth in the interior of $\Omega$ (\cite{Tay}) and we have \eqref{SD}
at our disposal for any $p$ such that $A$ is $p$-elliptic.  We prove \eqref{psq} first, and then \eqref{pAve} by a bootstrap argument so that no constants appearing depend on the smoothness of the coefficients of
$A$, $B$. 

Let $B_r(x)$ be ball in the interior of $\Omega$ with $r<\delta(x)/4$. Let $\varphi$ be a smooth cutoff function, with $\varphi = 1$ on $B_r(x)$ and vanishing outside of
$B_{2r}(x)$. Set $v = u\varphi$.
Then,
\begin{equation}\label{Lv}
{\mathcal L}v = u\mathcal L \varphi + A\nabla u\cdot \nabla \varphi + A^*\nabla u\cdot \nabla \varphi.
\end{equation}
Multiply both sides of \eqref{Lv} by $|v|^{p-2}\overline{v}$ and integrate by parts to obtain

\begin{align}\label{Lv1}
\int \nabla(|v|^{p-2}\overline{v}) \cdot A \nabla v \,dy  &=  \int (|v|^{p-2}\overline{v}) B\cdot \nabla v\,dy + \int \nabla(|v|^{p-2}\overline{v}u ) \cdot A\nabla \varphi \, dy \nonumber\\
& -\int |v|^{p-2}\overline{v} u \, B\cdot \nabla \varphi dy
- \int |v|^{p-2}\overline{v} A\nabla u\cdot \nabla \varphi \, dy  \nonumber\\
& - \int |v|^{p-2}\overline{v} A^*\nabla u\cdot \nabla \varphi \, dy 
 \nonumber\\
\end{align}
The real part of the left hand side of \eqref{Lv1} is bounded from below by $\lambda_p \int |v|^{p-2} |\nabla v|^2 \,dy$ by $p$-ellipticity and \eqref{SD}.  
The first of the five terms on the right hand side above has the bound
\begin{equation}\label{Lv3}
\left|\int (|v|^{p-2}\overline{v}) \cdot B\nabla v\,dy\right| \lesssim K r^{-1} \left( \int |v|^{p-2} |\nabla v|^2 \,dy\right)^{1/2} \left( \int |v|^p \,dy\right)^{1/2}
\end{equation}
where $K$ is as in \eqref{Bcond}.
The third term has the bound
\begin{equation}\label{Lv2}
\left|\int |v|^{p-2}\overline{v} u \, B\cdot \nabla \varphi dy\right| \lesssim  K \,r^{-2} \int_{B_{2r}(x)} |u|^p \,dy.
\end{equation}
In terms four and five, we use the fact that $\overline{v} \nabla u = \overline{u} \nabla v - |u|^2 \nabla \varphi$ to bound each of these integrals by a constant multiple of
\begin{equation}\label{gradu}
 \int |v|^{p-2} |u| |\nabla \varphi||\nabla v| + |v|^{p-2} |u|^2 |\nabla \varphi |^2 \,dy
\end{equation}
with constant depending on the ellipticity parameter $\Lambda$.
By Cauchy-Schwarz, this integral is bounded by terms like those in \eqref{Lv3} and \eqref{Lv2}.
After distributing the gradient, the second term has the following bound
\begin{equation}
\int |\nabla(|v|^{p-2}\overline{v})| |u| |\nabla \varphi| + |v|^{p-2} |\overline{v} \nabla u| |\nabla \varphi| \,dy
\end{equation}
Using Lemma \ref{l2.8} to compute the derivative in the first expression, the first term in the above integral is bounded by
\begin{equation}\label{4.15}
\int |v|^{p-2}| |\nabla v| |u| |\nabla \varphi| dy
\end{equation}
and the second expression, using the same trick as in \eqref{gradu} is bounded by 
\begin{equation}\label{4.16}
\int |v|^{p-2} |\nabla v| |u| |\nabla \varphi| + |v|^{p-2} |u|^2 |\nabla \varphi|^2 dy 
\end{equation}
Taking the real part of both sides of \eqref{Lv1} and combining all these estimates, and noting that the constants do not depend on the smoothness of the $A$ or $B$ gives \eqref{psq} for this $u$.
Observe, in addition, that the Sobolev embedding gives
\begin{equation}\label{emb}
\left(\dint_{B_{r}(x)} |u|^{\tilde p} \,dy\right)^{1/{\tilde p}} \lesssim  \left(\dint_{B_{2r}(x)} |v|^{\tilde p} \,dy\right)^{1/{\tilde p}} \lesssim  \left(r^2\dint_{B_{2r}(x)} |\nabla(|v|^{p/2-1} v)|^2 \,dy \right)^{1/p}
\end{equation}
where ${\tilde p} = \frac{pn}{n-2}.$

First we use \eqref{emb} and the argument for \eqref{psq} (see \eqref{Lv1} and the subsequent estimates) to obtain a reverse H\"older inequality for $u$. That is,
\begin{equation}\label{RH}
\left(\dint_{B_{r}(x)} |u|^{\tilde p} \,dy\right)^{1/{\tilde p}}
\lesssim
\left(\dint_{B_{\alpha r}(x)} |u|^{p} \,dy\right)^{1/{p}}
\end{equation}
Note that at this moment we have only proven \eqref{RH} for $\alpha = 2$; however, by adjusting the cutoff function $\varphi$, the entire argument can be done with
balls $B_{\alpha r}(x)$, for any $\alpha > 1$, and with a new constant which will depend on $\alpha$.
Iterating \eqref{RH} $k$ times gives
\begin{equation}\label{RHk}
\left(\dint_{B_{r}(x)} |u|^{p_k} \,dy\right)^{1/{p_k}}
\lesssim
\left(\dint_{B_{\alpha^k r}(x)} |u|^{2} \,dy\right)^{1/2}
\end{equation}
for $p_k = 2(\frac{n}{n-2})^k$, as long as $p_{k-1} < p$.

It remains to remove the assumption of smoothness of the coefficients of the operator, noting that the constants appearing in \eqref{psq} and \eqref{pAve} depend only on 
the $p$-ellipticity and
$K$ and not on any smoothness parameter.
Suppose  that $A$ satisfies the condition of $p$-ellipticity and $B$ satisfies \eqref{Bcond}, and let
$A_j$ and $B_j$ be smooth approximations, converging a.e. to $A$, $B$, respectively.  Let $u_j$ be the solution to $\mathcal L_ju_j = 0$ with
the same boundary data as $u$. We claim that the arguments of section 7 of \cite{KP2}  can be used to show that $u_j \rightarrow u$  strongly in $W^{1,2}$ on compact subsets of $\Omega$.  We note that the arguments of \cite{KP2} did not consider the convergence of  lower order terms, but this will follow in the same way from the strong convergence of
of $u_j$ to $u$ in $L^q$ for $q = 2n/(n-2)$. From the fact that \eqref{pAve} holds for $u_j$, and the strong 
convergence of $u_j$ to $u$ in $L^2(B_{2r})$, we see that the $L^q$ averages of $u_j$ are uniformly bounded. Passing to a subsequence if necessary, we have
weak convergence of the $u_j$ in $L^q$, and that weak limit must be $u$. This implies that the $L^q$ average of $u$ is bounded as well, i.e., we have
 \eqref{pAve} for $u$.
Now set $w_j =  |u_j(x)|^{p/2-1}u_j(x)$ and $w = |u(x)|^{p/2-1}u(x)$. From \eqref{psq} for $u_j$ and \eqref{pAve} for $u$, we have that $w_j$ is uniformly bounded in $W^{1,2}_\text{loc}(\Omega;\BBC).$  This gives weak convergence of $w_j$ to a limit, which again must be $w$. The weak convergence gives the uniform bound on $|\nabla w|^2$, and thus \eqref{psq} for $u$.

\end{proof}

\begin{lemma}\label{LpAve}
Let the matrix $A$ be $p$-elliptic for $p < 2$ and let $B$ have coefficients satisfying condition \eqref{Bcond} of Theorem \ref{Regularity}.
Suppose that $u$ is a $W^{1,2}_{loc}(\Omega;\BBC)$ solution to $\mathcal L$ in $\Omega$. Then, for any ball $B_r(x)$ with $r < \delta(x)/4$ and any $\varepsilon>0$
\begin{equation}\label{psq2}
r^2\dint_{B_r(x)} |\nabla u(y)|^2 |u(y)|^{p-2} dy \le C_\varepsilon \dint_{B_{2r}(x)} |u(y)|^p dy+\varepsilon\left( \dint_{B_{2r}(x)} |u(y)|^2 dy\right)^{p/2}
\end{equation}
and
\begin{equation}\label{pAve2}
\left(\dint_{B_{r}(x)} |u(y)|^2 dy\right)^{1/2}
\le  C_\varepsilon\left(\dint_{B_{2r}(x)} |u(y)|^p dy\right)^{1/p}+\varepsilon \left(\dint_{B_{2r}(x)} |u(y)|^2 dy\right)^{1/2}
\end{equation}
where the constants depend only $p$-ellipticity, $K$ of \eqref{Bcond} and $\varepsilon>0$.
In particular, $|u|^{(p-2)/2} u$ belongs to $W^{1,2}_{loc}(\Omega;\BBC).$
\end{lemma}

\begin{proof}
We assume that $A$ and $B$ are smooth, and remove this assumption later. We can therefore, as before, assume that we have a smooth solution to $\mathcal L$.
The proof of this lemma proceeds in a similar fashion as in \eqref{pAve}, but without introducing the cutoff
function $\varphi$. This will lead to presence of the two extra terms (with $\varepsilon$) on the right hand sides of \eqref{psq2} and \eqref{pAve2} which were not needed in \eqref{psq} and \eqref{pAve}. Unfortunately, if we attempt to proceed as in Lemma \ref{LpAvebig} with $v=u\varphi$ we will run into trouble estimating terms like 
\eqref{4.16} as $|v|^{p-2}\le|u|^{p-2}$ when $p\ge 2$, but in our case $p<2$ this does not hold and we run into problems near the boundary of supp $\varphi$ where $\varphi$ can be arbitrary small.

The fact that we do not introduce the cutoff function $\varphi$ does mean we have fewer terms to take care of but we do have to worry about one additional boundary integral and more importantly we also have to introduce an additional approximation to account for the fact that $p-2$ is negative.
To that end, define a cutoff function $\rho_{\delta}$ as follows.

\begin{equation}
\rho_\delta(s)=\begin{cases} \delta^{\frac{p-2}{2}},&\quad \mbox{if }0\le s\le\delta\\  s^{\frac{p-2}{2}},&\quad \mbox{if } s>\delta.   \end{cases}
\end{equation}

Following the proof of Lemma \ref{LpAve}, we multiply both sides of the equation $\mathcal L u=0$ by 
$\rho_\delta(|u|)^2 \overline{u}$ and integrate:

\begin{align}\label{rhov}
\int_{B_r(x)} \nabla(\rho^2_\delta(|u|)\overline{u}) \cdot A \nabla u \,dy  &=  \int_{B_r(x)} (\rho^2_\delta(|u|)\overline{u}) B\cdot \nabla u\,dy\\ &+ \int_{\partial{B_r(x)}} (\rho^2_\delta(|u|)\overline{u})\nu \cdot A\nabla u \, d\sigma(y). \nonumber
\end{align}
Here $\nu$ is the outer unit normal which for the ball is just $\frac{y-x}{r}$.

The left hand side of \eqref{rhov} splits into terms:
\begin{align}\label{2.35}
\int_{B_r(x)} \nabla(\rho^2_\delta(|u|)\overline{u}) \cdot A \nabla u \,dy &=
\delta^{p-2} \int_{B_{r}(x)\setminus E_\delta} \langle A \nabla u, \nabla u \rangle \,dy\\
&+ \int_{E_\delta\cap {B_r(x)}} A\nabla u \cdot \nabla (|u|^{p-2} \overline{u})\,dy\nonumber
\end{align}
where $E_\delta = \{|u| > \delta\}$.
Applying Theorem \ref{MTdiss} on the open set $E_\delta\cap {B_r(x)}$, we have that 
\begin{equation}\label{2.36}
\mathscr{R}e\int_{E_\delta\cap {B_r(x)}} A\nabla u \cdot \nabla (|u|^{p-2} \overline{u})\,dy \ge \lambda_p \int_{E_\delta\cap {B_r(x)}} |u|^{p-2}|\nabla u|^2\,dy
\end{equation}

Ultimately we will let $\delta \rightarrow 0$, and we will show that the integrals involving $B_r(x)\setminus E_\delta$ will tend to zero.
The main tool will be the following fact, established in \cite{Lan} for smooth functions $u$, namely
\begin{equation}
\delta^r \int_{B_{r}(x)\setminus E_\delta} |\nabla u|^2 dy \rightarrow 0
\end{equation}
for all $r> -1$.

We first take care of the boundary integral in \eqref{rhov}. Observe that \eqref{rhov}-\eqref{2.36} hold not only on the ball $B_r(x)$ but on any enlarged ball $B_{\alpha r}(x)$ for $1\le\alpha\le 3/2$. Hence if we write \eqref{rhov} for each such $\alpha$ and then average over the interval $[1,3/2]$ the last term of \eqref{rhov} will turn into a solid integral over the set $B_{3r/2}(x)\setminus B_r(x)$.

This and \eqref{2.35}-\eqref{2.36} then yields
\begin{align}\label{2.38}
\lambda_p \int_{E_\delta\cap {B_r(x)}} |u|^{p-2}|\nabla u|^2\,dy &\le  
\sup_{\alpha\in[1,3/2]}\left| \int_{B_{\alpha r}(x)} (\rho^2_\delta(|u|)\overline{u}) B\cdot \nabla u\,dy\right|\\
&+\nonumber \left|r^{-1}\int_{B_{3r/2}(x)\setminus B_r(x)} (\rho^2_\delta(|u|)\overline{u})\frac{(y-x)\cdot A\nabla u}{|y-x|}\,dy\right|   +o(1),
\end{align}
where $o(1)$ contains the integral over the complement of $E_\delta$, which tend to zero as $\delta\rightarrow 0$.

Each of the two terms on the right hand side of \eqref{2.38} will split into two integrals, one on $B_{3r/2}(x)\setminus E_\delta$ and one on $E_\delta\cap B_{3r/2}(x)$. Clearly,

\begin{align}\label{2.39}
\left| \int_{B_{\alpha r}(x)} (\rho^2_\delta(|u|)\overline{u}) B\cdot \nabla u\,dy\right|
&\lesssim r^{-1}\int_{E_\delta\cap B_{3r/2}(x)} |u|^{p-1}|\nabla u|\,dy\\\nonumber&+\delta^{p-1}r^{-1}\int_{B_{3r/2}(x)\setminus E_\delta}|\nabla u|\, dy,
\end{align}
and
\begin{align}\label{2.40}
 \left|r^{-1}\int_{B_{3r/2}(x)\setminus B_r(x)} (\rho^2_\delta(|u|)\overline{u})\frac{(y-x)\cdot A\nabla u}{|y-x|}\,dy\right| 
&\lesssim r^{-1}\int_{E_\delta\cap B_{3r/2}(x)} |u|^{p-1}|\nabla u|\,dy\\\nonumber&+\delta^{p-1}r^{-1}\int_{B_{3r/2}(x)\setminus E_\delta}|\nabla u|\, dy.
\end{align}
where the implied constants in \eqref{2.39} and \eqref{2.40} depend on $K$ and $\Lambda$ respectively.
The last terms in both inequalities  behave like $C\delta^{p-1}$, and will tend to zero as $\delta \rightarrow 0$ and hence can be written as $o(1)$ terms. By H\"older inequality we have for the first terms
on the right hand side of  \eqref{2.39}-\eqref{2.40}:
\begin{align}\label{2.41}
r^{-1}\int_{E_\delta\cap B_{3r/2}(x)} |u|^{p-1}|\nabla u|\,dy&\le r^{-1}\left(\int_{B_{3r/2}(x)}|u|^pdy\right)^{(p-1)/p}\left(\int_{B_{3r/2}(x)}|\nabla u|^pdy\right)^{1/p}\\\nonumber
&\le C_\varepsilon r^{-2}\int_{B_{3r/2}(x)}|u|^pdy +\varepsilon r^{p-2} \int_{B_{3r/2}(x)}|\nabla u|^pdy,
\end{align}
for any $\varepsilon>0$.

Putting all terms together therefore yields
\begin{equation}
\lambda_p r^2\int_{E_\delta} |v|^{p-2}|\nabla v|^2\,dy \le C_\varepsilon 
\int_{B_{3r/2}(x)} |u|^p \,dy +\varepsilon r^p \int_{B_{3r/2}(x)} |\nabla u|^p \,dy +o(1),
\end{equation}
where $o(1)$ contains all integrals over the complement of $E_\delta$, which tend to zero as $\delta \rightarrow 0$. The constant $C_\varepsilon$ depends on
$\Lambda$ and $K$ (and $\varepsilon$) but not
on the smoothness 
of $A$ and $B$.
Let $\delta$ tend to zero and obtain
\begin{equation}\label{almostdone}
r^2 \int_{\{y: \,u(y) \ne 0\} \cap B_r(x)} |u|^{p-2}|\nabla u|^2\,dy \le C'_\varepsilon 
 \int_{B_{2r}(x)} |u|^p \,dy+\varepsilon r^p \int_{B_{3r/2}(x)} |\nabla u|^p \,dy.
\end{equation}
Recalling the convention that $|u|^{p-2}|\nabla u|^2$ is taken to be zero whenever $\nabla u = 0$, the integral on the left hand side of 
\eqref{almostdone} can be taken on the set $B_r(x) \setminus \{y: u(y) =0, \nabla u(y) \ne 0\}$. However the measure of $\{y: u(y) =0, \nabla u(y) \ne 0\}$ is zero
and so we can conclude that (after introducing averages)
\begin{equation}\label{almostdone2}
r^2 \dint_{B_r(x)} |u|^{p-2}|\nabla u|^2\,dy \le C''_\varepsilon 
 \dint_{B_{2r}(x)} |u|^p \,dy+\varepsilon r^p \dint_{B_{3r/2}(x)} |\nabla u|^p \,dy
\end{equation}
holds. The final step is to use H\"older inequality and  Caccioppoli inequality (which is just  \eqref{psq} of Lemma \ref{LpAvebig} when $p=2$) for the last term of \eqref{almostdone2}. We have
\begin{equation}\label{almostdone3}
r^p \dint_{B_{3r/2}(x)} |\nabla u|^p \,dy\lesssim \left(r^2\dint_{B_{3r/2}(x)} |\nabla u|^2 \,dy\right)^{p/2}\lesssim  \left(\dint_{B_{2r}(x)} |u|^2 \,dy\right)^{p/2}.
\end{equation}
Finally, \eqref{almostdone2} and \eqref{almostdone3} combined yields  \eqref{psq2}.

The argument for  \eqref{pAve2} is similar to that for the case $p>2$. First, we have the analog of \eqref{RH}.  
\begin{equation}\label{RH2}
\left(\dint_{B_{r}(x)} |u|^{\tilde p} \,dy\right)^{1/{\tilde p}}
\le C_\varepsilon \left(\dint_{B_{\alpha r}(x)} |u|^{p} \,dy\right)^{1/{p}}+\varepsilon \left(\dint_{B_{\alpha r}(x)} |u|^{2} \,dy\right)^{1/{2}}
\end{equation}

The iteration scheme starts with a $p<2$ and stops when
we reach a $p_k$ greater than or equal to $2$. Each iteration will give us additional term of the form
$$\left(\dint_{B_{\alpha^i r}(x)} |u|^{2} \,dy\right)^{1/{2}}$$
on the right hand side for $i=1,2,\dots,k$ multiplied by a constant that can be as small as required, since at each step of the iteration $\varepsilon$ in \eqref{RH2} can be chosen independently of the previous choices.

Finally, since the constants in \eqref{psq2}
and \eqref{pAve2} do not depend on the smoothness of the coefficients of $A$ and $B$, arguments very similar to those in Lemma \ref{LpAvebig} allow us to pass to the limit and obtain these inequalities for solutions to $\mathcal L$ under the assumptions of the Lemma.

\end{proof}

The reverse H\"older inequalities of Lemmas \ref{LpAvebig} and \ref{LpAve} prove Theorem \ref{Regularity}.




%
%
%

\section{Carleson measures, nontangential maximal functions and $p$-adapted square functions}
\subsection{Nontangential maximal and square functions}
\label{SS:NTS}

On a domain of the form 
\begin{equation}\label{Omega-111}
\Omega=\{(x_0,x')\in\BBR\times{\BBR}^{n-1}:\, x_0>\phi(x')\},
\end{equation}
where $\phi:\BBR^{n-1}\to\BBR$ is a Lipschitz function with Lipschitz constant given by 
$L:=\|\nabla\phi\|_{L^\infty(\BBR^{n-1})}$, define for each point $x=(x_0,x')\in\Omega$
\begin{equation}\label{PTFCC}
\delta(x):=x_0-\phi(x')\approx\mbox{dist}(x,\partial\Omega).
\end{equation}
In other words, $\delta(x)$ is comparable to the distance of the point $x$ from the boundary of $\Omega$.

\begin{definition}\label{DEF-1}
A cone of aperture $a>0$ is a non-tangential approach region to the point $Q=(x_0,x') \in \partial\Omega$ defined as
\begin{equation}\label{TFC-6}
\Gamma_{a}(Q)=\{y=(y_0,y')\in\Omega:\,a|x_0-y_0|>|x'-y'|\}.
\end{equation}
\end{definition}

We require $1/a>L$, otherwise the  cone is too large and might not lie inside $\Omega$. 
But when $\Omega=\BBR^n_+$ all parameters $a>0$ may be considered.
Sometimes it is necessary to truncate $\Gamma(Q)$ at height $h$, in which case we write
\begin{equation}\label{TRe3}
\Gamma_{a}^{h}(Q):=\Gamma_{a}(Q)\cap\{x\in\Omega:\,\delta(x)\leq h\}.
\end{equation}

\begin{definition}\label{D:S}
For $\Omega \subset \mathbb{R}^{n}$ as above, the square function of some 
$u\in W^{1,2}_{loc}(\Omega; {\BBC})$ at $Q\in\partial\Omega$ relative 
to the cone $\Gamma_{a}(Q)$ is defined by
\begin{equation}\label{yrdd}
S_{a}(u)(Q):=\left(\int_{\Gamma_{a}(Q)}|\nabla u(x)|^{2}\delta(x)^{2-n}\,dx\right)^{1/2}
\end{equation}
and, for each $h>0$, its truncated version is given by 
\begin{equation}\label{yrdd.2}
S_{a}^{h}(u)(Q):=\left(\int_{\Gamma_{a}^{h}(Q)}|\nabla u(x)|^{2}\delta(x)^{2-n}\,dx\right)^{1/2}.
\end{equation}
\end{definition}

A simple application of Fubini's theorem gives 
\begin{equation}\label{SSS-1}
\|S_{a}(u)\|^{2}_{L^{2}(\partial\Omega)}\approx\int_{\Omega}|\nabla u(x)|^{2}\delta(x)\,dx.
\end{equation}

In [DPP], a ``$p$-adapted" square function was introduced in order to solve Dirichlet problems in the range
$1<p<2$. We shall use this method, and a similar $p$-adapted square function, but for both the ranges $p\geq2$ and 
$p<2$. In the following definition, when $p<2$ we use the convention that the expression $|\nabla u(x)|^{2} |u(x)|^{p-2}$ is zero whenever
$\nabla u(x)$ vanishes.

\begin{definition}\label{D:Sp}
For $\Omega \subset \mathbb{R}^{n}$, the $p$-adapted square function of 
$u\in W^{1,2}_{loc}(\Omega; {\BBC})$ at $Q\in\partial\Omega$ relative 
to the cone $\Gamma_{a}(Q)$ is defined by
\begin{equation}\label{yrddp}
S_{p,a}(u)(Q):=\left(\int_{\Gamma_{a}(Q)}|\nabla u(x)|^{2} |u(x)|^{p-2}\delta(x)^{2-n}\,dx\right)^{1/2}
\end{equation}
and, for each $h>0$, its truncated version is given by 
\begin{equation}\label{yrddp.2}
S_{p,a}^{h}(u)(Q):=\left(\int_{\Gamma_{a}^{h}(Q)}|\nabla u(x)|^{2}|u(x)|^{p-2}\delta(x)^{2-n}\,dx\right)^{1/2}.
\end{equation}
\end{definition}

We have shown that the expressions of the form $|\nabla u(x)|^{2} |u(x)|^{p-2}$, when $u$ is a solution of $\mathcal Lu=0$ are locally integrable and hence the definition of $S_p(u)$ makes sense for such $p$.

\begin{definition}\label{D:NT} 
For $\Omega\subset\mathbb{R}^{n}$ as above, and for 
a continuous $u: \Omega \rightarrow \mathbb C$, the nontangential maximal function ($h$-truncated nontangential maximal function) of $u$
 at $Q\in\partial\Omega$ relative to the cone $\Gamma_{a}(Q)$,
is defined by
\begin{equation}\label{SSS-2}
N_{a}(u)(Q):=\sup_{x\in\Gamma_{a}(Q)}|u(x)|\,\,\text{ and }\,\,
N^h_{a}(u)(Q):=\sup_{x\in\Gamma^h_{a}(Q)}|u(x)|.
\end{equation}
Moreover, we shall also consider a related version  of the above nontangential maximal function.
This is denoted by $\tilde{N}_{p,a}$ and is defined using $L^p$ averages over balls in the domain $\Omega$. 
Specifically, given $u\in L^p_{loc}(\Omega;{\BBC})$ we set
\begin{equation}\label{SSS-3}
\tilde{N}_{p,a}(u)(Q):=\sup_{x\in\Gamma_{a}(Q)}w(x)\,\,\text{ and }\,\,
\tilde{N}_{p,a}^{h}(u)(Q):=\sup_{x\in\Gamma_{a}^{h}(Q)}w(x)
\end{equation}
for each $Q\in\partial\Omega$ and $h>0$ where, at each $x\in\Omega$, 
\begin{equation}\label{w}
w(x):=\left(\dint_{B_{\delta(x)/2}(x)}|u(z)|^{p}\,dz\right)^{1/p}.
\end{equation}
\end{definition}

Above and elsewhere, a barred integral indicates an averaging operation. Observe that, given $u\in L^p_{loc}(\Omega;{\BBC})$, the function $w$ 
associated with $u$ as in \eqref{w} is continuous and $\tilde{N}_{p,a}(u)=N_a(w)$ 
everywhere on $\partial\Omega$.

The $L^2$-averaged nontangential maximal function was introduced in \cite{KP2} in connection with
the Neuman and regularity problem value problems. In the context of $p$-ellipticity, Theorem \ref{Regularity} shows that there is no difference between 
$L^2$ averages and $L^p$ averages, a fact which we record in the following proposition. 

\begin{proposition}\label{P3.5}
Suppose that $u$ is a $W^{1,2}_{loc}(\Omega;\BBC)$ solution to $\mathcal{L}=\mbox{\rm div} A(x)\nabla +B(x)\cdot\nabla$ in $\Omega$, where the matrix $A$ is
assume to be $p$-elliptic for $p \in (p_0, p'_0)$, and $B$ satisfies condition \eqref{Bcond}.
Then, for every $Q\in\partial\Omega$
\begin{equation}\label{pq}
\tilde{N}_{p,a}(u)(Q) \lesssim \tilde{N}_{q,a'}(u)(Q) \quad\forall p,q \in \left(p_0, \frac{p'_0n}{n-2}\right)
\end{equation}
when the aperture parameters of the cones satisfy $a < a'$ and $p\le q$ or $q\ge2$. 

When $a<a'$, $q<2$, $p>q$ and $p,q \in \left(p_0, \frac{p'_0n}{n-2}\right)$ we have a weaker estimate
\begin{equation}\label{pq2}
\tilde{N}_{p,a}(u)(Q) \lesssim \tilde{N}_{q,a'}(u)(Q) +\varepsilon \tilde{N}_{2,a'}(u)(Q),
\end{equation}
for all $\varepsilon>0$.

Finally, for any aperture parameters $a, a'$ in the appropriate range, and for all $p,q$ in the same range as 
\eqref{pq},
\begin{equation}
\|\tilde{N}_{p,a}(u)\|_{L^r(\partial \Omega)} \approx \|\tilde{N}_{q,a'}(u)\|_{L^r(\partial \Omega)}
\end{equation}
for all $r >0$.
\end{proposition}

\begin{proof} Clearly, \eqref{pq} and \eqref{pq2} follows from \eqref{RHthm1} (as we have noted earlier the integral over $B_{2r}$ in \eqref{RHthm1} can be replaced by $B_{\alpha r}$ for any $\alpha>1$ and hence the cones $\Gamma_{a'}(Q)$ used on the right hand side of \eqref{pq} and \eqref{pq2} can be just little bit larger that the cone $\Gamma_{a}(Q)$ used in definition of $\tilde{N}_{p,a}(u)$ on the left hand side of \eqref{pq} and \eqref{pq2}).

Looking at the equivalence in the $L^r$ norm, consider for the moment still the case when $a'>a$. When $p<q$ or $q\ge2$ 
\begin{equation}\label{3.16}
\|\tilde{N}_{p,a}(u)\|_{L^r(\partial \Omega)}\lesssim \|\tilde{N}_{q,a'}(u)\|_{L^r(\partial \Omega)}
\end{equation}
follows from \eqref{pq}. Otherwise by \eqref{pq2} applied to $p=2$ yields
\begin{equation}\label{3.17}
\|\tilde{N}_{2,a}(u)\|_{L^r(\partial \Omega)}\lesssim \|\tilde{N}_{q,a'}(u)\|_{L^r(\partial \Omega)}+\varepsilon  \|\tilde{N}_{2,a'}(u)\|_{L^r(\partial \Omega)}.
\end{equation}

The norm equivalence of the nontangential maximal functions $\tilde{N}_{p,a}$ and $\tilde{N}_{p,a'}$ for different aperture parameters $a,a'$ requires only a classical real-variable
argument using the level sets $\{\tilde{N}_{p,a}(u)>\lambda\}$ and $\{\tilde{N}_{p,a'}(u)>\lambda\}$. Hence always
\begin{equation}\label{3.16a}
\|\tilde{N}_{p,a'}(u)\|_{L^r(\partial \Omega)}\lesssim \|\tilde{N}_{p,a}(u)\|_{L^r(\partial \Omega)}.
\end{equation}
Combining this with \eqref{3.17} implies that 
\begin{equation}\label{3.17a}
\|\tilde{N}_{2,a}(u)\|_{L^r(\partial \Omega)}\lesssim \|\tilde{N}_{q,a'}(u)\|_{L^r(\partial \Omega)}+\varepsilon  \|\tilde{N}_{2,a}(u)\|_{L^r(\partial \Omega)},
\end{equation}
which by choosing $\varepsilon>0$ sufficiently small then gives
\begin{equation}\label{3.19}
\|\tilde{N}_{2,a}(u)\|_{L^r(\partial \Omega)}\lesssim \|\tilde{N}_{q,a'}(u)\|_{L^r(\partial \Omega)}.
\end{equation}
Combining this with \eqref{3.16} for any $p\in  \left(p_0, \frac{p'_0n}{n-2}\right)$ and $a<a''<a'$ we then have
\begin{equation}\label{3.20}
\|\tilde{N}_{p,a}(u)\|_{L^r(\partial \Omega)}\lesssim \|\tilde{N}_{2,a''}(u)\|_{L^r(\partial \Omega)}\lesssim \|\tilde{N}_{q,a'}(u)\|_{L^r(\partial \Omega)}
\end{equation}
as desired. Hence \eqref{3.16} holds for any $p,q$ in our range as long as $a'>a$.

The reverse inequality to \eqref{3.16} can be obtained using \eqref{3.16a} by choosing $a''<a$. This gives
$$\|\tilde{N}_{q,a'}(u)\|_{L^r(\partial \Omega)}\lesssim \|\tilde{N}_{q,a''}(u)\|_{L^r(\partial \Omega)}$$
and then since $a''<a$ by \eqref{3.16}
$$\|\tilde{N}_{q,a''}(u)\|_{L^r(\partial \Omega)}\lesssim \|\tilde{N}_{p,a}(u)\|_{L^r(\partial \Omega)}.$$

\end{proof}


\subsection{Carleson measures}
\label{SS:Car} 

We begin by recalling the definition of a Carleson measure in a domain $\Omega$ as in \eqref{Omega-111}. 
For $P\in{\BBR}^n$, define the ball centered at $P$ with the radius $r>0$ as
\begin{equation}\label{Ball-1}
B_{r}(P):=\{x\in{\BBR}^n:\,|x-P|<r\}.
\end{equation}
Next, given $Q \in \partial\Omega$, by $\Delta=\Delta_{r}(Q)$ we denote the surface ball  
$\partial\Omega\cap B_{r}(Q)$. The Carleson region $T(\Delta_r)$ is then defined by
\begin{equation}\label{tent-1}
T(\Delta_{r}):=\Omega\cap B_{r}(Q).
\end{equation}

\begin{definition}\label{Carleson}
A Borel measure $\mu$ in $\Omega$ is said to be Carleson if there exists a constant $C\in(0,\infty)$ 
such that for all $Q\in\partial\Omega$ and $r>0$
\begin{equation}\label{CMC-1}
\mu\left(T(\Delta_{r})\right)\leq C\sigma(\Delta_{r}),
\end{equation}
where $\sigma$ is the surface measure on $\partial\Omega$. 
The best possible constant $C$ in the above estimate is called the Carleson norm 
and is denoted by $\|\mu\|_{\mathcal C}$.
\end{definition}

In all that follows we now assume that the coefficients of the matrix $A$ and $B$ of the elliptic operator $\mathcal{L}=\mbox{div} A(x)\nabla +B(x)\cdot\nabla$  
satisfies the following natural conditions.  
First, we assume that the entries $A_{ij}$ of $A$ are in ${\rm Lip}_{loc}(\Omega)$ and the entries of $B$ are $L^\infty_{loc}(\Omega)$. 
Second, we assume that
\begin{equation}\label{CarA}
d\mu(x)=\sup_{B_{\delta(x)/2}(x)}[|\nabla A|^2+|B|^2]\delta(x) \,dx
\end{equation}
is a Carleson measure in $\Omega$. Sometimes, and for certain coefficients of $A$, we will
assume that their Carleson norm $\|\mu\|_{\mathcal{C}}$ is sufficiently small. 
Crucially we have the following result. 

\begin{theorem}\label{T:Car}
Suppose that $d\nu=f\,dx$ and $d\mu(x)=\left[\sup_{B_{\delta(x)/2}(x) }|f|\right]dx$. Assume that
$\mu$ is a Carleson measure. Then there exists a finite 
constant $C=C(L,a)>0$ such that for every $u\in L^{p}_{loc}(\Omega;{\BBC})$ one has
\begin{equation}\label{Ca-222}
\int_{\Omega}|u(x)|^p\,d\nu(x)\leq C\|\mu\|_{\mathcal{C}} 
\int_{\partial\Omega}\left(\tilde{N}_{p,a}(u)\right)^p\,d\sigma.
\end{equation}
Furthermore, consider $\Omega={\mathbb R}^n_+$ where  $\mu$ and $\nu$ are measures as above supported in $\Omega$ and $\delta(x_0,x')=x_0$. Let
 $h:{\mathbb R}^{n-1}\to {\mathbb R}^+$ be a Lipschitz function with Lipschitz norm $L$
and 
$$\Omega_h=\{(x_0,x'):x_0>h(x')\}.$$
Then for any $\Delta\subset {\mathbb R}^{n-1}$ with $\sup_{\Delta} h\le \mbox{diam}(\Delta)/2$ we have
\begin{equation}\label{Ca-222-x}
\int_{\Omega_h\cap T(\Delta)}|u(x)|^p\,d\nu(x)\leq C\|\mu\|_{\mathcal{C}} 
\int_{\partial\Omega_h\cap T(\Delta)}\left(\tilde{N}_{p,a,h}(u)\right)^p\,d\sigma.
\end{equation}
Here for a point $Q=(h(x'),x')\in\partial\Omega_h$ we define
\begin{equation}
\tilde{N}_{p,a,h}(u)(Q) = \sup_{\Gamma_a(Q)}w,\label{eq-Nh}
\end{equation}
where
\begin{equation}\label{TFC-6x}
\Gamma_{a}(Q)=\Gamma_{a}((h(x'),x'))=\{y=(y_0,y')\in\Omega:\,a|h(x')-y_0|>|x'-y'|\}
\end{equation}
and the $L^p$ averages $w$ are defined by \eqref{w} where the distance $\delta$ is taken with respect to the domain $\Omega={\mathbb R}^n_+$.
\end{theorem}

\begin{proof} For the first part of the claim let 
$$\Omega=\bigcup_{i}{\mathcal O}_i$$
be a Whitney decomposition of $\Omega$ and assume that the Whitney sets ${\mathcal O}_i$ are such that
for any $x\in {\mathcal O}_i$ we have ${\mathcal O}_i\subset B_{\delta(x)/2}(x)$. Also, $|{\mathcal O}_i|\approx |B_{\delta(x)/2}(x)|$.
It follows that on each ${\mathcal O}_i$ we have
$$\int_{{\mathcal O}_i}|u(x)|^p\,d\nu(x)\leq \left[\sup_{{\mathcal O}_i}|f|\right] \int_{{\mathcal O}_i}|u(x)|^p\,dx.$$
By the definition \eqref{w} for $w$ it follows that for any $y\in {{\mathcal O}_i}$ we have
$$\int_{{\mathcal O}_i}|u(x)|^p\,d\nu(x)\lesssim \left[\sup_{B_{\delta(y)/2}(y) }|f|\right] w(y)^p|{\mathcal O}_i|.$$
From this
\begin{equation}
\int_{{\mathcal O}_i}|u(x)|^p\,d\nu(x)\lesssim \int_{{\mathcal O}_i} w(y)^pd\mu(y).\label{eq_W}
\end{equation}
Summing over all $i$ we get
$$\int_{\Omega}|u(x)|^p\,d\nu(x)\lesssim \int_{\Omega} w(y)^pd\mu(y)\lesssim \|\mu\|_{\mathcal C}\int_{\partial\Omega}N_a(w)^p\,d\sigma,$$
where the last inequality follows from the usual inequality for the Carleson measure. Since $\tilde{N}_{p,a}(u)=N_a(w)$ the claim follows.

The second claim has a similar argument. Because $h$ is Lipschitz and $\sup_{\Delta} h\le \mbox{diam}(\Delta)/2$ it follows that
$$\Omega_h\cap T(\Delta)=\bigcup_{i}{\mathcal O}_i$$
where again ${\mathcal O}_i$ are Whitney sets with respect to the original domain $\Omega={\mathbb R}^n_+$. It follows that for each $i$ once again \eqref{eq_W} holds. Summing over $i$ we get
$$\int_{\Omega_h\cap T(\Delta)}|u(x)|^p\,d\nu(x)\lesssim \int_{\Omega_h\cap T(\Delta)} w(y)^pd\mu(y)\lesssim \|\mu\|_{\mathcal C}\int_{\partial\Omega_h\cap T(\Delta)}N_{a,h}(w)^p\,d\sigma,$$
where the last inequality is a standard estimate for a Carleson measure on a Lipschitz domain.
\end{proof}

\subsection{Pullback Transformation}
\label{SS:PT}
The Carleson measure conditions, \eqref{CarA}, on the coefficients of  $\mathcal L$
are compatible with a useful change of variables described in this subsection. 

For a domain $\Omega$ as in \eqref{Omega-111}, consider a mapping 
$\rho:\mathbb{R}^{n}_{+}\to\Omega$ appearing in works of Dahlberg, Ne\v{c}as, 
Kenig-Stein and others, defined by
\begin{equation}\label{E:rho}
\rho(x_0, x'):=\big(x_0+P_{\gamma x_0}\ast\phi(x'),x'\big),
\qquad\forall\,(x_0,x')\in\mathbb{R}^{n}_{+},
\end{equation}
for some positive constant $\gamma$. Here $P$ is a nonnegative function 
$P\in C_{0}^{\infty}(\mathbb{R}^{n-1})$ and, for each $\lambda>0$,  
\begin{equation}\label{PPP-1a}
P_{\lambda}(x'):=\lambda^{-n+1}P(x'/\lambda),\qquad\forall\,x'\in{\mathbb{R}}^{n-1}.
\end{equation}
Finally, $P_{\lambda}\ast\phi(x')$ is the convolution
\begin{equation}\label{PPP-lambda}
P_{\lambda}\ast\phi(x'):=\int_{\mathbb{R}^{n-1}}P_{\lambda}(x'-y')\phi(y')\,dy'. 
\end{equation}
Observe that $\rho$ extends up to the boundary of ${\BBR}^{n}_{+}$ and maps one-to-one from 
$\partial {\BBR}^{n}_{+}$ onto $\partial\Omega$. Also for sufficiently small $\gamma\lesssim L$ 
the map $\rho$ is a bijection from $\overline{\mathbb{R}^{n}_{+} }$ onto $\overline\Omega$ 
and, hence, invertible. 

For $u\in W^{1,2}_{loc}(\Omega;\BBC)$ that solves $\mathcal{L}u=0$ in $\Omega$ with Dirichlet 
datum $f$ consider $v:=u\circ\rho$ and $\tilde{f}:=f\circ\rho$. The change of variables 
via the map $\rho$ just described implies that $v\in W^{1,2}_{loc}(\mathbb{R}^{n}_{+};{\BBC})$ 
solves a new elliptic PDE of the form
\begin{equation}\label{ESvvv}
0=\mbox{div} (\tilde{A}(x)\nabla v)+\tilde{B}(x)\cdot\nabla v,
\end{equation}
with boundary datum $\tilde{f}$ on $\partial \mathbb{R}^{n}_{+}$. Hence, solving a boundary value 
problem for $u$ in $\Omega$ is equivalent to solving a related boundary value problem for $v$ in 
$\mathbb{R}^{n}_{+}$. Crucially, if the coefficients of the original system are such that \eqref{CarA} 
is a Carleson measure, then the coefficients of $\tilde{A}$ and $\tilde{B}$ satisfy an analogous 
Carleson condition in the upper-half space. If, in addition, the Carleson norm of \eqref{CarA} 
is small and $L$ (the Lipschitz constant for the domain $\Omega$) is also small, then the Carleson 
norm for the new coefficients $\tilde{A}$ and $\tilde{B}$ will be correspondingly small. Hence the 
map $\rho$ allows us to assume that the
domain is $\Omega={\BBR}^n_+$.

Moreover, this transformation also preserves $p$-ellipticity.
\section{The $L^p$-Dirichlet problem}
\label{S3}
 
When an operator $\mathcal L$ is as in Theorem \ref{S3:T1} is uniformly elliptic in the sense of \eqref{EllipAA},
the Lax-Milgram lemma can be applied and guarantees the existence of weak solutions.
That is, 
given any $f\in \dot{B}^{2,2}_{1/2}(\partial\Omega;{\BBC})$, the homogenous space of traces of functions in $\dot{W}^{1,2}(\Omega;{\BBC})$, there exists a unique (up to a constant)
$u\in \dot{W}^{1,2}(\Omega;{\BBC})$ such that $\mathcal{L}u=0$ in $\Omega$ and ${\rm Tr}\,u=f$ on $\partial\Omega$. We call these solutions \lq\lq  energy solutions" and use them to define the notion of solvability of the $L^p$ Dirichlet problem.

\begin{definition}\label{D:Dirichlet} 
Let $\Omega$ be the Lipschitz domain introduced in \eqref{Omega-111} and fix an integrability exponent 
$p\in(1,\infty)$. Also, fix an aperture parameter $a>0$. Consider the following Dirichlet problem 
for a complex valued function $u:\Omega\to{\BBC}$:
\begin{equation}\label{E:D}
\begin{cases}
0=\partial_{i}\left(A_{ij}(x)\partial_{j}u\right) 
+B_{i}(x)\partial_{i}u 
& \text{in } \Omega,
\\[4pt]
u(x)=f(x) & \text{ for $\sigma$-a.e. }\,x\in\partial\Omega, 
\\[4pt]
\tilde{N}_{2,a}(u) \in L^{p}(\partial \Omega), &
\end{cases}
\end{equation}
where the usual Einstein summation convention over repeated indices ($i,j$ in this case) 
is employed. 

We say the Dirichlet problem \eqref{E:D} is solvable for a given $p\in(1,\infty)$ if  there exists a
$C=C(p,\Omega)>0$ such that
for all boundary data
$f\in L^p(\partial\Omega;{\BBC})\cap B^{2,2}_{1/2}(\partial\Omega;{\BBC})$ the unique \lq\lq energy solution"
satisfies the estimate
\begin{equation}\label{y7tGV}
\|\tilde{N}_{2,a} (u)\|_{L^{p}(\partial\Omega)}\leq C\|f\|_{L^{p}(\partial\Omega;{\BBC})}.
\end{equation}
\end{definition}

\noindent{\it Remark.}  Given $f\in\dot{B}^{2,2}_{1/2}(\partial\Omega;{\BBC})\cap L^p(\partial\Omega;{\BBC})$
the corresponding energy solution constructed above is unique (since the decay of $L^p$ eliminates constant solutions). As the space 
$\dot{B}^{2,2}_{1/2}(\partial\Omega;{\BBC})\cap L^p(\partial\Omega;{\BBC})$ is dense in 
$L^p(\partial\Omega;{\BBC})$ for each $p\in(1,\infty)$, it follows that there exists a 
unique continuous extension of the solution operator
$f\mapsto u$
to the whole space $L^p(\partial\Omega;{\BBC})$, with $u$ such that $\tilde{N}_{2,a} (u)\in L^p(\partial\Omega)$ 
and the accompanying estimate $\|\tilde{N}_{2,a} (u) \|_{L^{p}(\partial \Omega)} 
\leq C\|f\|_{L^{p}(\partial\Omega;{\BBC})}$ being valid.

Moreover, we shall establish in the appendix that under the assumptions of Theorem \ref{S3:T1} for any $f\in L^p(\partial \Omega;\mathbb C)$ the corresponding solution $u$ constructed by the continuous extension attains the datum $f$ as its boundary values in the following sense.
Consider the average $\tilde u:\Omega\to \mathbb C$ defined by
$$\tilde{u}(x)=\dint_{B_{\delta(x)/2}(x)} u(y)\,dy,\quad \forall x\in \Omega.$$
Then 
\begin{equation}
f(Q)=\lim_{x\to Q,\,x\in\Gamma(Q)}\tilde u(x),\qquad\text{for a.e. }Q\in\partial\Omega,
\end{equation}
where the a.e. convergence is taken with respect to the ${\mathcal H}^{n-1}$ Hausdorff measure on $\partial\Omega$.

\vskip2mm
Let us make some observations that explain the structural assumptions we have made in Theorem \ref{S3:T1}.
As we have already stated it suffices to formulate the result in the case $\Omega={\mathbb R}^n_+$ by using the pull-back map introduced above. Since Theorem \ref{S3:T1} requires that the coefficients have {\it small} Carleson norm this puts a restriction on the size of the Lipschitz constant $L=\|\nabla\phi\|_{L^\infty}$ of the map $\phi$ that defines the domain $\Omega$ in \eqref{Omega-111}.
The constant $L$ will have also to be small (depending on $\lambda_p$, $\Lambda$, $n$ and $p$). \vglue1mm

For technical reasons in the proof we also need that all coefficients $A_{0j}$, $j=0,1,\dots,n-1$ are real. This can be ensured as follows. When $j>0$ observe that we have 
\begin{equation}
\partial_0([\mathscr{I}m\,A_{0j}]\partial_j u)=\partial_j([\mathscr{I}m\,A_{0j}]\partial_0 u)+(\partial_0 [\mathscr{I}m\,A_{0j}])\partial_ju-([\partial_j \mathscr{I}m\,A_{0j}])\partial_0u\label{eqSWAP}
\end{equation}
which allows to move the imaginary part of the coefficient $A_{0j}$ onto the coefficient $A_{j0}$ at the expense of two (harmless) first order terms. This does not work for the coefficient $A_{00}$. Instead we make the following observation. 

Suppose that the measure \eqref{CarA}  associated to an operator $ \mathcal L = \partial_{i}\left(A_{ij}(x)\partial_{j}\right)  +B_{i}(x)\partial_{i}$  is Carleson. 
Consider a related operator 
$ \tilde{\mathcal L} = \partial_{i}\left(\tilde{A}_{ij}(x)\partial_{j}\right) 
+\tilde{B}_{i}(x)\partial_{i}$,
where $\tilde A = \alpha{A}$ and $\tilde B=\alpha{B} - (\partial_{i}\alpha){A}_{ij}\partial_j$, and
$\alpha\in L^\infty(\Omega)$ is a complex valued function such that $|\alpha(x)|\ge \alpha_0>0$ and $|\nabla\alpha|^2x_0$ is a Carleson measure. 
 
Observe that a weak solution $u$ to $\tilde{\mathcal L}u=0$ is also a weak solution to $\mathcal Lu=0$ and that
the new coefficients of $\tilde A$ and $\tilde B$ also satisfy a Carleson measure condition as in \eqref{CarA}, from the assumption on $\alpha$. We will only require that the coefficient ${\tilde A}_{00}$ is real but we may as well ensure for simplicity that it equals to $1$. Clearly, if we choose $\alpha = {A}_{00}^{-1}$, then the new operator $\tilde L$ will have this property. When ${A}_{00}$ (and hence $\alpha$) is real, then  $\tilde{A}$. Similarly, if ${A} $ is $p$-elliptic and $\mathscr{I}m\,{A}_{00}$ is 
sufficiently small (depending on the ellipticity constants), then $\tilde A$ will also be  $p$-elliptic. However, if $\mathscr{I}m\,\alpha$ is not small,
the $p$-ellipticity, after multiplication of ${A}$ by $\alpha$ may not be preserved. Thus, we assume in our main result (Theorem \ref{S3:T1}) the 
$p$-ellipticity of the new matrix  $\tilde A$ which has all coefficients $\tilde A_{0j}$, $j=0,1,\dots,n-1$ real, as this is not implied in the general case from the $p$-ellipticity of the original matrix $A$.

The solutions to the Dirichlet problem in the infinite domain $\BBR^n_+$ will be obtained as a limit of solutions in
infinite strips $\Omega^h=\{x=(x_0, x')) \in \BBR \times {\BBR}^{n-1}: 0 < x_0 < h\}$. We define them as follows.

\begin{definition}\label{D:DirichletStrip} 
Let $\Omega = {\BBR}^n_+$, and let $\Omega^h$ be the infinite strip 
$$\Omega^h=\{x=(x_0, x')) \in \BBR \times {\BBR}^{n-1}: 0 < x_0 < h\},$$ and let
$p\in(1,\infty)$. Also, fix an aperture parameter $a>0$. 
Let $u$ be a complex valued function $u:\Omega\to{\BBC}$ such that

\begin{equation}\label{E:D-strip}
\begin{cases}
0=\partial_{i}\left(A_{ij}(x)\partial_{j}u\right) 
+B_{i}(x)\partial_{i}u 
& \text{in } \Omega^h,
\\[4pt]
u(x_0,x') = 0, & \text{for all }x_0 \geq h,
\\[4pt]
u(x)=f(x) & \text{ for $\sigma$-a.e. }\,x\in\partial\Omega, 
\\[4pt]
\tilde{N}_{2,a}(u) \in L^{p}(\partial \Omega), &
\end{cases}
\end{equation}
where the usual Einstein summation convention over repeated indices ($i,j$ in this case) 
is employed. 

We say the Dirichlet problem \eqref{E:D-strip} is solvable for a given $p\in(1,\infty)$ if  there exists a
$C=C(p,\Omega)>0$ such that
for all boundary data
$f\in L^p(\partial\Omega;{\BBC})\cap B^{2,2}_{1/2}(\partial\Omega;{\BBC})$ we have that $u\big|_{\Omega^h}$ is the unique ``energy solution" to
\begin{equation}\label{E:D-energy}
\begin{cases}
0=\partial_{i}\left(A_{ij}(x)\partial_{j}u\right) 
+B_{i}(x)\partial_{i}u 
& \text{in } \Omega^h,
\\[4pt]
u(x_0,x') = 0, & \text{for } x_0=h
\\[4pt]
u(x)=f(x) & \text{ for $\sigma$-a.e. }\,x\in\partial\Omega, 
\end{cases}
\end{equation}

and satisfies the estimate
\begin{equation}\label{y7tGV-2}
\|\tilde{N}_{2,a} (u)\|_{L^{p}(\partial\Omega)}\leq C\|f\|_{L^{p}(\partial\Omega;{\BBC})}.
\end{equation}
\end{definition}

\bigskip

We now give the proof of Theorem \ref{S3:T1}. We will establish the solvability of the Dirichlet problem \eqref{E:D-strip} assuming that the coefficients of $A$ and $B$ are smooth, applying the results of sections 5 and 6. The constants
will not depend on the degree of smoothness or on the width of the strip.
Then, a limiting argument proves Theorem \ref{S3:T1} for smooth coefficients. Finally, we consider smooth approximations of $\mathcal L$, and another
limiting argument gives Theorem \ref{S3:T1}.

\begin{proof}

Let  $u_h$ be the energy solution in $\Omega^h$ as in Definition \ref{D:DirichletStrip}.
As follows from Corollary \ref{S4:C2}
\begin{equation}\label{S3:L4:E00bbSS}
\lambda'_p\iint_{\BBR^n_+}|\nabla u_h|^{2}|u_h|^{p-2}x_0\,dx'\,d x_0 
\leq\int_{\BBR^{n-1}}|f(x')|^{p}\,dx'
+C\|\mu'\|_{\mathcal{C}}\int_{\BBR^{n-1}}\left[\tilde{N}_{p,a}(u_h)\right]^{p}\,dx'.
\end{equation}

We shall momentarily assume finiteness of the quantities involving nontangential maximal functions and square functions needed to apply 
Corollary \ref{S3:C7}, and return to this point later.  
Thus we have
\begin{equation}
\int_{\BBR^{n-1}}\left[\tilde{N}_{p,a}(u_h)\right]^{p}\,dx'
\leq C_1\int_{\BBR^{n-1}}|f(x')|^{p}\,dx'
+C_2\|\mu'\|_{\mathcal{C}}\int_{\BBR^{n-1}}\left[\tilde{N}_{p,a}(u_h)\right]^{p}\,dx'.
\end{equation}
Here the constants $C_1,C_2$ depend on $p,\,\lambda_p\,,\Lambda,\,n$ and $\|\mu\|_C$. It follows that for 
$$\|\mu'\|_C<\frac1{2C_2}$$
we have that
\begin{equation}\label{goodgrief}
\int_{\BBR^{n-1}}\left[\tilde{N}_{p,a}(u_h)\right]^{p}\,dx'
\leq 2C_1\int_{\BBR^{n-1}}|f(x')|^{p}\,dx'.
\end{equation}

We now consider the limit of $u_h$, as $h \to \infty$. 
The uniform Lax-Milgram estimate on $\|\nabla u_h\|_{L^2({\BBR}^n_+)}$ by $\|f\|_{\dot B_{1/2}^{2,2}}$, and the fact that $\text{Tr}(u_h) = f$, gives a weakly convergent subsequence to some $u$ with
$\|\nabla u\|_{L^2({\BBR}^n_+)} \leq C \|f\|_{\dot B_{1/2}^{2,2}}$ and $\text{Tr}(u) = f$.
This subsequence is therefore strongly convergent to $u$ in $L^2_{loc}({\BBR}^n_+)$
It follows that the $L^2$ averages $w_h$ of $u_h$ converge locally and uniformly to $w$, the $L^2$ averages of $u$ in $C_{loc}({\BBR}^n_+)$.

Let $\Gamma_k(x')$ be the doubly truncated cone $\Gamma(x') \cap \{1/k < x_0 < k\}$.
Define $$\tilde{N}_k(u)(x') = \sup_{y \in \Gamma_k(x')} |w(y)|,$$ and with $\tilde{N}_k(u_h)(x')$ defined analogously.
Then we have
$$\tilde{N}_k(u_h)(x') \to \tilde{N}_k(u)(x') \quad \text{uniformly on compact subsets $K \subset {\BBR}^{n-1}.$}$$
Finally, using \eqref{goodgrief}, this give on each such set $K$,
$$\|\tilde{N}_k(u)\|_{L^p(K)} = \lim_{h \to \infty}   \|\tilde{N}_k(u_h)\|_{L^p(K)} \leq C\|f\|_{L^p({\BBR}^{n-1})}. 
$$
The constant $C$ in the estimate above is independent of $K$ and $k$, so that taking the supremum in each of $k$ and $K$ gives the desired 
estimate assuming the coefficients are smooth.
Finally, we approximate our coefficients by smooth functions, and the passage from the smooth coefficient case requires a further argument that 
mirrors the limiting process above using truncations of cones.

We now address the finiteness requirements of Corollary \ref{S3:C7}, in two separate cases. In the a priori estimate $\|\tilde{N}_{p,a}(u_h)\|_{L^p({\BBR}^{n-1})}< c$, the 
constant $c$ is allowed to depend on measures of smoothness of the coefficients or on the truncation parameter $h$, as the only property used to obtain estimate  (\ref{goodgrief})  is the fact
that $\|\tilde{N}_{p,a}(u_h)\|_{L^p({\BBR}^{n-1})}$ is finite.

Consider first the case $p \geq 2$. First,
$$\|\tilde{N}_{p,a}(u_h)\|^2_{L^2({\BBR}^{n-1})}  \lesssim \|S_{2,a}(u_h)\|^2_{L^2({\BBR}^{n-1})} \lesssim \int_{\Omega^h} |\nabla u_h|^2 dx < \infty.$$
By interpolation with the Agmon-Douglis-Nirenberg (\cite{ADN}) $L^{\infty}$ bound for solutions to smooth systems, it follows that
$\|\tilde{N}_{p,a}(u_h)\|_{L^p({\BBR}^{n-1})}< \infty.$ Since $p \geq p'$, this suffices to apply Corollary \ref{S3:C7}.

Now suppose that $p < 2$.
In this case, we need that both $\|\tilde{N}_{p,a}(u_h)\|_{L^p({\BBR}^{n-1})}< \infty$ and $\|S_{p',a}(u_h)\|^p_{L^p({\BBR}^{n-1}} < \infty.$
We shall use an extrapolation argument based on an method in \cite{DKV} of obtaining $L^{2-\varepsilon}$ estimates of nontangential maximal functions
from $L^2$ estimates on sawtooth domains. See also \cite{DHM}, where this technique was used to get solvability of the $L^p$ Dirichlet problem for elliptic systems for
$2 - \varepsilon < p < 2$.
In particular, the argument of \cite{DKV}, reproduced in section 6 of \cite{DHM} for systems and hence valid in our setting, gives that $\|\tilde{N}_{2,a}(u_h)\|_{L^{p_0}({\BBR}^{n-1})}< \infty$
for $p_0=2 - \varepsilon$ and hence the same is true for $\|\tilde{N}_{p_0,a}(u_h)\|_{L^{p_0}({\BBR}^{n-1})}$. The quantity $\varepsilon$ depends on the constant $C_2$ in the
$L^2$ norm inequality between the nontangential maximal function and the square function $S_2$.  We now observe that this gives
$\|S_{p'_0,a}(u_h)\|^2_{L^{p_0}({\BBR}^{n-1})} < \infty$ as well: Use the fact that, pointwise, $$S_{p'_0,a}(u_h) < C_{\eta}S_{p_0,a}(u_h) + \eta N(u_h)$$  where $N$ is the pointwise
maximal function, and $\eta$ is as small as we wish, together with Corollary \ref{S4:C2} to bound $\|S_{p_0,a}(u_h)\|_{L^{p_0}({\BBR}^{n-1})}$ by 
$\|\tilde{N}_{p_0,a}(u_h)\|_{L^{p_0}({\BBR}^{n-1})}$.

Once these two quantities are finite, Corollary \ref{S3:C7} applies and we obtain (\ref{goodgrief}), and hence the same estimate for $u$, for $p_0=2-\varepsilon$ 
and a constant $C_{2-\epsilon}$.

The very same argument, now invoking the $L^{p_0}$ estimate gives an $L^{p_0 - \varepsilon'}$ estimate where $\varepsilon'$ now depends on $C_{2-\varepsilon}$. In other words,
we apply the same argument as \cite{DKV} but starting from known estimates for the nontangential maximal function in  $L^{p_0}$ instead of $L^2$. There is no 
difference in the structure of the argument. We can continue this bootstrapping as long as we stay in the range of $p$-ellipticity and as long as we can
be sure that we are moving by an amount $\varepsilon$ which is not getting smaller at each step. This last point is assured by the fact that the constants $C_p$ in the 
$L^p$ norm inequalities (\ref{goodgrief}) only depend on the Lipschitz constants, $p$-ellipticity and the Carleson measure norm of the coefficients. Thus, for fixed $q<2$ where the operator
is $q$-elliptic, the constants $C_p$ for $q<p<2$ are uniformly bounded.

\end{proof}

\section{Estimates for the $p$-adapted square function $S_p(u)$}
\label{S4}

In this, and the next, section we make the assumption that the coefficients of $A$ and $B$ are smooth, in order to ensure the finiteness of 
$L^p$ norms of the nontangential maximal function.

We fix an $h>1$, and an infinite strip $\Omega^h$ defined above, and let $u$ be an energy solution to \eqref{E:D-energy}, extended to be zero above 
height $h$.  
In this section we establish a one sided estimate of the $p$-adapted square function of $u$ in terms of boundary data and 
its nontangential maximal function, with constants independent of $h$.

\begin{lemma}\label{S3:L4}
Let $u:\Omega \to  \BBC$ be as above, 
with the Dirichlet boundary datum $f\in \dot{B}^{2,2}_{1/2}(\partial\Omega;{\BBC}) \cap  L^{p}(\partial \Omega;{\BBC})$. Assume that 
$A$ is $p$-elliptic and smooth in $\BBR^n_+$ with $A_{00} =1$ and $A_{0j}$ real and that the measure $\mu'$ defined as in \eqref{Car_hatAAB} is Carleson. Then there exists a constant $C=C(\lambda_p,\Lambda,p,n)$ such that for all $r>0$
\begin{align}\label{S3:L4:E00}
&\hskip -0.20in
p\frac{\lambda_p}2\iint_{[0,r/2]\times\partial\Omega}|u|^{p-2}|\nabla u|^{2}x_0\,dx'\,d x_0 
+\frac{2}{r}\iint_{[0,r]\times\partial \Omega} |u(x_0,x')|^{p}\,dx'\,dx_0 
\nonumber\\[4pt]
&\hskip 0.20in
\leq\int_{\partial\Omega}|f(x')|^{p}\,dx' 
+\int_{\partial\Omega}|u(r,x')|^{p}\,dx'
+C\|\mu'\|_{\mathcal{C}}\int_{\partial\Omega}\left[\tilde{N}^{r}_{p,a}(u)\right]^{p}\,dx'.
\end{align}
\end{lemma}

\begin{proof} To proceed, fix an arbitrary $y'\in\partial\Omega\equiv{\mathbb{R}}^{n-1}$, and consider first an $r\leq h$.
Pick 
a smooth cutoff function $\zeta$ which is $x_0-$independent and satisfies
\begin{equation}\label{cutoff-F}
\zeta= 
\begin{cases}
1 & \text{ in } B_{r}(y'), 
\\
0 & \text{ outside } B_{2r}(y').
\end{cases}
\end{equation}
Moreover, assume that $r|\nabla \zeta| \leq c$ for some positive constant $c$ independent of $y'$. 
We begin by considering the integral quantity 
\begin{equation}\label{A00}
\mathcal{I}:=\mathscr{R}e\,\iint_{[0,r]\times B_{2r}(y')}A_{ij}\partial_{j}u 
\partial_{i}(|u|^{p-2}\overline{u})x_0\zeta\,dx'\,dx_0
\end{equation}
with the usual summation convention understood. With $\chi=x_0\zeta$ we have by Theorem \ref{MTdiss}
for some $\lambda_p>0$
\begin{equation}\label{cutoff-AA}
\mathcal{I}\geq{\lambda_p}\iint_{[0,r]\times B_{2r}}|u|^{p-2}|\nabla u|^2 x_0\zeta\,dx'\,dx_0,
\end{equation}
where we agree henceforth to abbreviate $B_{2r}:=B_{2r}(y')$ whenever convenient.

The idea now is to integrate by parts the formula for $\mathcal I$ in order 
to relocate the $\partial_i$ derivative. This gives 
\begin{align}\label{I+...+IV}
\mathcal{I}
&= \mathscr{R}e\,\int_{\partial\left[(0,r)\times B_{2r}\right]} 
A_{ij}\partial_{j}u|u|^{p-2}\overline{u}x_0\zeta\nu_{x_i}\,d\sigma 
\nonumber\\[4pt]
&\quad - \mathscr{R}e\,\iint_{[0,r]\times B_{2r}}\partial_{i}\left(A_{ij} 
\partial_{j}u\right)|u|^{p-2}\overline{u}x_0\zeta\,dx'\,dx_0 
\nonumber\\[4pt]
&\quad - \mathscr{R}e\,\iint_{[0,r]\times B_{2r}}A_{ij}\partial_{j}{u}|u|^{p-2}\overline{u}\partial_{i}x_0\zeta\,dx'\,dx_0 
\nonumber\\[4pt]
&\quad - \mathscr{R}e\,\iint_{[0,r]\times B_{2r}}A_{ij}\partial_{j}u|u|^{p-2}\overline{u}x_0\partial_{i}\zeta\,dx'\,dx_0
\nonumber\\[4pt]
&=:I+II+III+IV,
\end{align}
where $\nu$ is the outer unit normal vector to $(0,r)\times B_{2r}$. The boundary term $I$ does not vanish
only on the set $\{r\}\times B_{2r}$ and only when $i=0$. This gives
\begin{equation}\label{cutoff-BBB}
I=\mathscr{R}e\,\int_{\{r\}\times B_{2r}} 
A_{0j}\partial_{j}u|u|^{p-2}\overline{u}x_0\zeta\,d\sigma 
\end{equation}

As $u$ is a weak solution of  $\mathcal L u=0$ in $\Omega$, we use the equation to transform $II$ into
\begin{equation}\label{cutoff-CCC}
II=\mathscr{R}e\,\iint_{[0,r]\times B_{2r}}B_{i}(\partial_{i}u)|u|^{p-2}\overline{u}x_0\zeta\,dx'\,dx_0.
\end{equation}
To further estimate this term we use H\"older's inequality, the Carleson condition for the term $B$ and 
Theorem~\ref{T:Car} in order to write
\begin{align}\label{TWO-TWO}
|II| &\leq\left(\iint_{[0,r]\times B_{2r}}\left|B\right|^{2} 
|u|^{p} x_0\zeta\,dx'\,dx_0\right)^{1/2}  
\cdot\left(\iint_{[0,r]\times B_{2r}}|u|^{p-2}|\nabla u|^{2}x_0\zeta\,dx'\,dx_0\right)^{1/2} 
\nonumber\\[4pt]
&\leq C(\lambda_p,\Lambda,p,n)\left(\|\mu'\|_{\mathcal{C}}\int_{B_{2r}} 
\left[\tilde{N}^r_{p,a}(u)\right]^{p}\,dx'\right)^{1/2}\cdot\mathcal{I}^{1/2}. 
\end{align}

As $\partial_ix_0=0$ for $i>0$ the term $III$ is non-vanishing only for $i=0$. We further split this term
by considering the cases when $j=0$ and $j>0$. This yields, since $A_{00} =1$, 

\begin{align}\label{u6fF}
III_{\{j=0\}} &=-\mathscr{R}e\,\iint_{[0,r]\times B_{2r}}\partial_{0}{u}|u|^{p-2}\overline{u}\zeta\,dx'\,dx_0 \nonumber\\
&=-\frac1p\iint_{[0,r]\times B_{2r}}  \partial_{0}(|u|^{p})\zeta\,dx'\,dx_0 \\
&=-\frac{1}{p}\int_{B_{2r}} |u|^p(r,x')\zeta\,dx' + \frac{1}{p}\int_{B_{2r}} |u|^p(0,x')\zeta\,dx' \nonumber
\end{align}

When $j>0$ we first use the fact that $A_{0j}$ is real and hence the expression $\mathscr{R}e\, [{A}_{0j}\, (\partial_{j}u)|u|^{p-2}\overline u]=p^{-1}A_{0j}\partial_j(|u|^p)$. Then we 
reintroduce $1=\partial_0x_0$  and integrate by parts moving the $\partial_0$ derivative
\[\begin{split}
III_{\{j \neq 0\}}
&= -\mathscr{R}e\, \iint_{[0,r]\times B_{2r}} {A}_{0j}\, \partial_{j}u|u|^{p-2} \, \overline{u} \,  \zeta \,dx'\,dx_0 \\
& -p^{-1}\, \iint_{[0,r]\times B_{2r}} {A}_{0j}\, \partial_{j}(|u|^{p}) \, \left(\partial_{0} x_0\right) \zeta \,dx'\,dx_0 \\
&= p^{-1}\int_{B_{2r}} {A}_{0j} \partial_{j}(|u|^p)(r,x') r \zeta \,dx' + p^{-1}\iint_{[0,r]\times B_{2r}} \partial_{0}{A}_{0j} \partial_{j}(|u|^p) x_0 \zeta \,dx'\,dx_0 \\
&+ p^{-1}\iint_{[0,r]\times B_{2r}} {A}_{0j} \partial^2_{0j}(|u|^p) x_0 \zeta \,dx'\,dx_0  \\
&= III_{1} + III_{2} + III_{3}.
\end{split}\]
We note that $III_{1} = - I_{\{j \neq 0\}}$.

In the third term $III_3$ we switch the order of derivatives $\partial^2_{0j} = \partial^2_{j0}$ and take further integration by parts with respect to $\partial_j$.
\[\begin{split}
III_{3}
&= - p^{-1}\iint_{[0,r]\times B_{2r}} \partial_{j} {A}_{0j} \partial_{0}(|u|^p) x_0 \zeta \,dx'\,dx_0 \\
&\quad - p^{-1}\iint_{[0,r]\times B_{2r}} {A}_{0j}\partial_{0}(|u|^p) x_0 (\partial_{j}\zeta) \,dx'\,dx_0 = III_{31} + III_{32}.
\end{split}\]

The terms $III_2$ and $III_{31}$ are of the same type as $II$ we have handled earlier and hence have the same estimate
\[
III_{2} + III_{31} \leq C(\lambda_p,\Lambda, p,n) \left( \|\mu'\|_{\mathcal{C}} \int_{ B_{2r}} \left[\tilde{N}^r_{p,a}(u)\right]^{p}\,dx'\right)^{1/2} \cdot \mathcal I^{1/2}
\]

We add up all terms we have so far to obtain
\begin{equation}\label{square01}\begin{split}
\mathcal I
&\leq p^{-1}\int_{B_{2r}}  \partial_{0}(|u|^p)(r,x') r \zeta \,dx' \\
&\quad  - {p}^{-1}\int_{B_{2r}} |u|^p(r,x')\zeta\,dx' + {p}^{-1}\int_{B_{2r}} |u|^p(0,x')\zeta\,dx'  \\
&\quad + C(\lambda_p,\Lambda,p,n) \|\mu'\|_{\mathcal{C}} \int_{B_{2r}} \left[\tilde{N}^{r}_{p,a}(u)\right]^p(u) \,dx' +\frac12\mathcal I \\
&\quad + III_{32}  + IV.
\end{split}\end{equation}

We have used the arithmetic-geometric inequality for expression bounding the term $II$ in \eqref{TWO-TWO} as well as for similar terms $III_2$ and $III_{31}$.

To obtain a global version of \eqref{square01}, consider a sequence of disjoint boundary balls 
$(B_r(y'_k))_{k\in\mathbb N}$ such that $\cup_{k}B_{2r}(y'_k) $ covers $\partial\Omega={\BBR}^{n-1}$ 
and consider a partition of unity $(\zeta_{k})_{k\in\mathbb N}$ subordinate to this cover. That is, 
assume $\sum_k \zeta_{k} = 1$ on ${\BBR}^{n-1}$ and each $\zeta_{k}$ is supported in $B_{2r}(y'_k)$. 
Write $IV_k$ for each term as the last expression in \eqref{I+...+IV} corresponding to 
$B_{2r}=B_{2r}(y'_k)$. Given that $\sum_k \partial_i\zeta_{k} = 0$ for each $i$, by summing 
\eqref{square01} over all $k$'s gives $\sum_{k} IV_k= 0$. The same observation applies to the terms arising in $III_{32}$.
It follows that
\begin{align}\label{square02}
&\hskip -0.20in 
\frac{\lambda_p}{2}\iint_{[0,r]\times{\BBR}^{n-1}}|\nabla u|^2|u|^{p-2}\,x_0\,dx'\,dx_0 \leq
\nonumber\\[4pt]
&\hskip 0.20in
p^{-1}\int_{{\BBR}^{n-1}}  \partial_{0}(|u|^p)(r,x') r  \,dx' \nonumber\\
&\quad  - {p}^{-1}\int_{{\BBR}^{n-1}} |u|^p(r,x')\,dx' + {p}^{-1}\int_{{\BBR}^{n-1}} |u|^p(0,x')\,dx'  
\nonumber\\[4pt]
&\quad +C\|\mu'\|_{\mathcal{C}}\int_{{\BBR}^{n-1}}\left[\tilde{N}^{r}_{p,a}(u)\right]^p\,dx'.
\end{align}
We have established \eqref{square02} for $r \leq h$, but we 
now observe that \eqref{square02} holds also for $r>h$, as $u = 0$ when $r \geq h$.
To see this, note that when $r=h$, the second term on the right hand side of the inequality is negative, and the third term is zero.
From this, \eqref{S3:L4:E00} follows by integrating \eqref{square02} in $r$ on
$[0,r']$ and dividing by $r'$.
\end{proof}

Lemma~\ref{S3:L4}, and its proof, yields several important corollaries. 

\begin{corollary}\label{S4:C2} 
Under the assumptions of Lemma \ref{S3:L4} we have for such $u$:
\begin{equation}\label{S3:L4:E00bb}
\lambda'_p\iint_{\BBR^n_+}|\nabla u|^{2}|u|^{p-2}x_0\,dx'\,d x_0 
\leq\int_{\BBR^{n-1}}|f(x')|^{p}\,dx'
+C\|\mu'\|_{\mathcal{C}}\int_{\BBR^{n-1}}\left[\tilde{N}_{p,a}(u)\right]^{p}\,dx'.
\end{equation}
Furthermore, under the same assumptions, if $g:{\mathbb R}^{n-1}\to{\mathbb R}^+$ is a Lipschitz function with small Lipschitz norm for any $\Delta\subset{\mathbb R}^{n-1}$ such that $\sup_\Delta g\le d/2$ where $d=\mbox{diam}(\Delta)$ we also have the following local estimate
\begin{equation}\label{eq5.15}
\iint_{\Omega_g\cap T(\Delta)}|\nabla u|^{2}|u|^{p-2}\delta_g(x)\,dx 
\leq C\int_{2\Delta}\left(|u(g(x'),x')|^p+(1+\|\mu\|_{\mathcal{C}})\left[\tilde{N}^{2d}_{p,a,g}(u)\right]^{p}\right)\,dx'.
\end{equation}
Here $\tilde{N}^{2d}_{p,a,g}$ is the truncated version of the nontangential maximal function defined in \eqref{eq-Nh}
with respect to the domain $\Omega_g=\{x_0>g(x')\}$ and $\delta_g$ measures the distance of a point to the boundary of $\Omega_g$. 
\end{corollary}

\begin{proof} The first claim follows immediately from Lemma~\ref{S3:L4} by taking $r\to \infty$ since $\int_{\partial\Omega}|u(r,x')|^{p}\,dx' = 0$ when $r>h$.

The second claim can be seen as follows. In the case when the function $g(x')=0$, 
one proceeds exactly as in the proof of Lemma~\ref{S3:L4} above until  
\eqref{square01}. Then instead of summing over different balls $\Delta=B_r$ covering $\mathbb R^{n-1}$ we estimate the terms $III_{32}$ and $IV$. Both of these terms are of the same type and can be bounded (up to a constant) by
\begin{equation}\label{eq5.16}
\iint_{[0,r]\times B_{2r}}|\nabla u||u|^{p-1}x_0|\partial_T\zeta|dx'dx_0,
\end{equation}
where $\partial_T\zeta$ denotes any of the derivatives in the direction parallel to the boundary. Recall that $\zeta$ is a smooth cutoff function equal to $1$ on $B_r$ and $0$ outside $B_{2r}$. In particular, we may assume $\zeta$ to be of the form $\zeta=\eta^2$ for another smooth function $\eta$ such that $|\nabla_T\eta|\le C/r$. By Cauchy-Schwarz \eqref{eq5.16} can be further estimated by
\begin{equation}\label{eq5.17}
\left(\iint_{[0,r]\times B_{2r}}|\nabla u|^2|u|^{p-2}x_0(\eta)^2dx'dx_0\right)^{1/2}\left(\iint_{[0,r]\times B_{2r}}|u|^{p}x_0|\nabla_T\eta|^2dx'dx_0\right)^{1/2}
\end{equation}
\begin{equation}
\lesssim{\mathcal I}^{1/2}\left(\frac1r\iint_{[0,r]\times B_{2r}}|u|^pdx'dx_0\right)^{1/2}\le \varepsilon{\mathcal I}+C_\varepsilon\int_{B_{2r}}\left[\tilde{N}^r_{p,a,g}(u)\right]^{p}\,dx'.\nonumber
\end{equation}
In the last step we have used the AG-inequality and a straightforward estimate of the solid integral $|u|^p$ by the $p$-averaged nontangential maximal function. Substituting \eqref{eq5.17} into \eqref{square01} the estimate \eqref{eq5.15} follows by integrating in $r$ over $[0,r']$ and dividing by $r'$ exactly as done above. We note that, by the second part of Theorem \ref{T:Car}, we can use in the estimates $\tilde{N}_{p,a}$ defined as in \eqref{eq-Nh}.

In the general case, for $g$ Lipschitz with small constant, we use the pullback map to again work on ${\mathbb R}^n_+$. Recall that we working in
the infinite strip $\Omega^h$, but this will transform under this mapping.  Let $G(x')$ be the image of 
the line $x_0 =h$ under this pullback. Instead of integrating on $[0.r]$ in
Lemma~\ref{S3:L4}, we integrate first in $x'$ and then on $[0.r'(x')]$, where $r' = \text{min}(r, G(x'))$.
Since $u$ vanishes at $(x', G(x'))$, there will still be no contribution from boundary integrals when integrating derivatives $\partial_j u$, $j<n-1$, and
the argument goes through as before.

\end{proof}

\begin{lemma}\label{LGL2} 
Let $\Omega=\BBR^n_+$ and assume $u$ be the energy solution of \eqref{E:D-energy}  Assume that 
$A$ is $p$-elliptic and smooth in $\BBR^n_+$ with $A_{00} =1$ and $A_{0j}$ real and that the measure $\mu$ defined as in \eqref{Car_hatAA} is Carleson.

Consider any $b>a>0$. Then for each $\gamma\in(0,1)$ there exists a constant $C(\gamma)>0$ 
such that $C(\gamma,a,b)\to 0$ as $\gamma\to 0$ and with the property that for each $\nu>0$ and 
each energy solution $u$ of \eqref{E:D} there holds 
\begin{align}\label{eq:gl2}
&\hskip -0.20in 
\left|\Big\{x'\in {\BBR}^{n-1}:\,S_{p,a}(u)(x')>\nu,\, \tilde{N}_b(u)(x')\le\gamma \nu\Big\}\right|
\nonumber\\[4pt] 
&\hskip 0.50in
\quad\le C(\gamma)\left|\big\{x'\in{\BBR}^{n-1}:\,{S}_{p,b}(u)(x')>\nu/2\big\}\right|.
\end{align}
\end{lemma}

\begin{proof} 
We observe that $\big\{x'\in{\BBR}^{n-1}:\,{S}_{p,b}(u)(x')>\nu/2\}$ is an open 
subset of ${\BBR}^{n-1}$. When this set is empty, or is all of ${\BBR}^{n-1}$, 
estimate \eqref{eq:gl2} is trivial, so we focus on the case when the set in question is 
both nonempty and proper. Granted this, we may consider a Whitney decomposition $(\Delta_i)_{i\in I}$ 
of it, consisting of open cubes in ${\mathbb{R}}^{n-1}$. Let $F_\nu^i$ be the set appearing on the 
left-hand side of \eqref{eq:gl2} intersected with $\Delta_i$.  Let $r_i$ be the diameter of $\Delta_i$. Due to the nature of the Whitney decomposition there exists a point $p'\in 2\Delta_i$ such that ${S}_{p,b}(u)(p')<\nu/2$.
From this and the fact that $b>a$ it follows that for all $x'\in F^i_\nu$ we have
$$S^{d}_{p,a}(u)(x')>\nu/2,$$
where $S^{d}_{p,a}$ is the truncated version of the square function at some height $d\approx r_i$, where the precise nature of relation between $d$ and $r_i$ depends on the apertures $a$ and $b$. 

For some $a<c<b$ consider the domain
$$\Omega_c=\bigcup_{x'\in F^i_\nu} \Gamma_c(x');$$
this is a Lipschitz domain with Lipschitz constant $1/c$. Observe that $F^i_\nu\subset \partial\Omega_c$.
It follows that
$$|F^i_\nu|\le \frac{2^p}{\nu^p}\int_{F^i_\nu}\left[S^{d}_{p,a}(u)(x')\right]^p\,dx'\lesssim \nu^{-p}\iint_{\Omega_c\cap T(\Delta_i)}|\nabla u|^2|u|^{p-2}\delta_c\,dx,$$
where $\delta_c$ measures the distance of a point to the boundary $\partial\Omega_c$. It follow by \eqref{eq5.15}
$$|F^i_\nu|\lesssim \nu^{-p}\int_{\partial\Omega_c\cap T(2\Delta_i)}\left(\left|u\big|_{\partial\Omega_c}\right|^p+(1+\|\mu\|_{\mathcal{C}})\left[\tilde{N}^{2d}_{p,a,c}(u)\right]^{p}\right)\,d\sigma,$$
where $\tilde{N}^{2d}_{p,a,c}$ is defined using nontangential cones with aperture $a$ with vertices on $\partial\Omega_c$. Due to the fact that each of these cones is contained in one of the cones $\Gamma_b(x')$
for some $x'\in F^i_\nu$ (as $c<b$) and on $F^i_\nu$: $\tilde{N}_b(u)(x')\le\gamma \nu$ we also have 
$\tilde{N}^{2d}_{p,a,c}(u)\le\gamma\nu$ everywhere on $\partial\Omega_c$. Thus we can conclude
$$|F^i_\nu|\lesssim \nu^{-p}\int_{\partial\Omega_c\cap T(2\Delta_i)}\left|u\big|_{\partial\Omega_c}\right|^p\,d\sigma+(1+\|\mu\|_{\mathcal{C}})\nu^{-p}(\gamma\nu)^p|2\Delta_i|.$$
We still need to deal with the first term on the righthand side. We convert this term into a solid integral by 
averaging $c$ over the interval $[a,b]$. Such solid integral has a trivial estimate by
$$C\int_{\partial\Omega_b\cap T(2\Delta_i)}\left[\tilde{N}^{2d}_{p,a,b}(u)\right]^{p}d\sigma\lesssim (\gamma\nu)^p|2\Delta_i|.$$
Hence using that the surface measure is doubling we finally get
$$|F^i_\nu|\le C(a,b,\|\mu\|_{\mathcal C})\gamma^p.$$
Summing over all $i$ yields \eqref{eq:gl2} as desired.
\end{proof}

We will require a localized version of Lemma \ref{LGL2} in order to deal with integrals of square functions and nontangential maximal
functions that are not a priori finite.

\begin{lemma}\label{LocalGoodL}
Let $u$, defined with respect to $\Omega^h$, and $A$, $a,b$, be as in Lemma \ref{LGL2}. Fix $R \geq h$ and consider
the boundary ball $\Delta_R \subset {\BBR}^{n-1}$. Let $p > q > 1.$ for any $q$ such that $A$ is $q$-elliptic.
Let 
$$\nu_0^p = C \dint_{\Delta_{2R}} N_b^p(u) dx',$$
where $C$ is a constant depending only on dimension (calculated in the proof below).
Then for each $\gamma\in(0,1)$ there exists a constant $C(\gamma)>0$ 
such that $C(\gamma,a,b)\to 0$ as $\gamma\to 0$ and with the property that for each $\nu>\nu_0$ 
\begin{align}\label{eq:gl2-2}
&\hskip -0.20in 
\left|\Big\{x'\in \Delta_R :\,S_{q,a}(u)(x')>\nu,\, \tilde{N}_b(u)(x')\le\gamma \nu\Big\}\right|
\nonumber\\[4pt] 
&\hskip 0.50in
\quad\le C(\gamma)\left|\big\{x'\in \Delta_R:\,{S}_{q,b}(u)(x')>\nu/2\big\}\right|.
\end{align}

\end{lemma} 
\begin{proof}
By Corollary \ref{S4:C2},  $\|S_{q, b}(u)\|_{L^q(\Delta_R)} \lesssim \|N_b(u)\|_{L^q(\Delta_{2R})}.$ Therefore,

\begin{align}\label{Sqwithnu}
\hskip -0.20in 
\big| \Delta_R \cap \{S_q > \nu/2\} \big|    &\lesssim  \nu^{-q} \|N_b(u)\|^q_{L^q(\Delta_{2R})}\\
&
\lesssim \nu^{-q} \|N_b(u)\|^{q/p}_{L^p(\Delta_{2R})} \big| \Delta_{2R} \big|^{1-q/p}  \nonumber\\
&\lesssim    C_{\varepsilon}\nu^{-p} \int_{\Delta_{2R}} (N_b(u))^p + \varepsilon \big|\Delta_{R}\big|.
\end{align}
Choosing $\varepsilon = 1/4$, which determines $C_{\varepsilon}$,  and we now fix $C = 4C_{\varepsilon}$ in the definition of $\nu_0$.
This implies that for any $\nu>\nu_0$, we have that
$$\big| \Delta_R \cap \{S_{q,b} > \nu/2\} \big| < 1/2 \big| \Delta_R\big|.$$
Thus, there exists a Whitney decomposition of  $\Delta_R \cap \{S_{q,b} > \nu/2\}$ into open cubes $\Delta_i$ with the property that
$2\Delta_i \cap \Delta_R$ contains a point for which $S_{q,b}(u) < \nu/2.$
From this point on, the proof proceeds as in Lemma \ref{LGL2}.

\end{proof}

\begin{corollary}\label{S4:C1} Under the assumption of Lemma \ref{LGL2}, for any $q \geq p >1$ and $a>0$ there exists a finite 
constant $C=C(\lambda_p,\Lambda,p,q,a,\|\mu\|_{\mathcal C},n)>0$ such that 
\begin{equation}\label{S3:C7:E00oo=s}
\|S_{p,a}(u)\|_{L^{q}({\BBR}^{n-1})}\le C\|\tilde{N}_{p,a}(u)\|_{L^{q}({\BBR}^{n-1})}.
\end{equation}
The statement also holds for any $q > 0$, 
provided we know a priori that
\newline
$\|S_{p,a}(u)\|_{L^{q}({\BBR}^{n-1})} < \infty.$

\end{corollary}
\begin{proof}

This is a consequence of the good-$\lambda$ inequality established above and the equivalence (\cite{CMS}) of $p$-adapted square functions with different aperture in any $L^q$ norm.

When $q \geq p$, and $M$ is large,  
 $$ \int_0^M   \nu^{q-1}  \big| \Delta_R \cap \{S_{p,a}(u) > \nu\}\big| d\nu \leq C(M) \int_0^M   \nu^{p-1}  \big| \Delta_R \cap \{S_{p,a}(u) > \nu\}\big| d\nu.
 $$

 By Corollary \ref{S4:C2}, and the fact that the coefficients are smooth, the right hand side is bounded. Therefore, the left hand side is also bounded, with a constant that
 may depend on $M$. 
 
 Now we multiply the good-$\lambda$ inequality of
 Lemma \ref{LocalGoodL} by  $\nu^{p-1}$ and integrate separately over $(0, \nu_0)$ and $(\nu_0, M)$.
 This gives

$$ \|S_{p,a}(u)\|_{L^{q}({\Delta_{R}})}  \le C\|\tilde{N}_{p,a}(u)\|_{L^{q}({\Delta_{2R}})}, $$

\noindent after taking the limit as $M \to \infty$. 

The estimate \eqref{S3:C7:E00oo=s} follows after summing over a decomposition of ${\BBR}^{n-1}$ into balls of size $R$ and adding the local estimates.

When $q<p$, the local good-$\lambda$ inequality is not available, which is why we need the additional a priori estimate on the finiteness of 
$\|S_{p,a}(u)\|_{L^{q}({\BBR}^{n-1})}.$

\end{proof}

\section{Bounds for the nontangential maximal function by the $p$-adapted square function}
\label{SS:43}

As before, we work on $\Omega=\BBR^n_+$ and we assume that the matrix $A$ is $p$-elliptic.  
Our aim in this section is to establish the converse of the inequality in Corollary~\ref{S4:C1}. 
The approach necessarily differs from the usual argument in the real scalar elliptic case due to the fact 
that certain estimates, such as interior H\"older regularity of a weak solution, are unavailable for 
the complex coefficient case. Hence, alternative arguments bypassing such difficulties 
must be devised. We use here an adaptation of the approach developed for elliptic systems in \cite{DHM}.

Since any scalar complex valued PDE can be written as a real skew-symmetric system, the theorem of \cite{DHM}
can be applied here and we have the following result (c.f. Proposition 5.8 of \cite{DHM}).

\begin{proposition}\label{S3:C7a} 
Let $u$ be an arbitrary energy solution of \eqref{E:D} in $\Omega=\BBR^n_{+}$.
Assume that $A$ is elliptic and the measure $\mu$ defined as in 
\eqref{Car_hatAA} is Carleson with norm $\|\mu\|_{\mathcal C}<\infty$.
Then for any $q>0$ and  $a>0$ there exists a finite 
constant $C=C(\lambda,\Lambda,q,a,\|\mu\|_{\mathcal C},n)>0$ such that 
\begin{equation}\label{S3:C7:E00oo-xx}
\|\tilde{N}_{2,a}(u)\|_{L^{q}({\BBR}^{n-1})}\le C\|S_{2,a}(u)\|_{L^{q}({\BBR}^{n-1})}.
\end{equation}
\end{proposition}

\begin{corollary}\label{S3:C7} Let $\Omega=\BBR^n_+$ and assume $u$ is a solution of \eqref{E:D-energy}.  Assume that 
$A$ is $p$-elliptic and smooth in $\BBR^n_+$ with $A_{00} =1$ and $A_{0j}$ real and that the measure $\mu$ defined as in \eqref{Car_hatAA} is Carleson.
Then for any $q>0$ and  $a>0$ there exists a finite 
constant $C=C(\lambda_p,\Lambda,p,q,a,\|\mu\|_{\mathcal C},n)>0$ such that 
\begin{equation}\label{S3:C7:E00oo}
\|\tilde{N}_{p,a}(u)\|_{L^{q}({\BBR}^{n-1})}\le C\|S_{p,a}(u)\|_{L^{q}({\BBR}^{n-1})},
\end{equation}
provided that a priori $\|\tilde{N}_{p,a}(u)\|_{L^{q}({\BBR}^{n-1})} < \infty$. If the dual exponent $p'>q$ we also have to assume that $\|S_{p',a}(u)\|_{L^{q}({\BBR}^{n-1})} < \infty$.

\end{corollary}

\begin{proof} Using H\"older's inequality we have for any $x'\in{\mathbb R}^{n-1}$
\begin{align}
&\hskip-5mm\left[S_{2,a}(u)(x')\right]^2=\iint_{\Gamma_a(x')}|\nabla u|^{2/p}| u|^{1-2/p}
|\nabla u|^{2/p'}| u|^{1-2/p'}x_0\,dx'\,dx_0\nonumber\\
&\le \left(\iint_{\Gamma_a(x')}|\nabla u|^2|\ u|^{p-2}
x_0\,dx'\,dx_0\right)^{1/p}\left(\iint_{\Gamma_a(x')}|\nabla u|^2| u|^{p'-2}
x_0\,dx'\,dx_0\right)^{1/p'}\label{eq-pf3a}\\
\nonumber &\le S_{p,a}(u)(x')S_{p',a}(u)(x').
\end{align}
It follows by Proposition \ref{S3:C7a} 
\begin{equation}\label{S3:C7:E00oo-xx2}
\|\tilde{N}_{2,a}(u)\|^2_{L^{q}({\BBR}^{n-1})}\le C\|S_{p,a}(u)\|_{L^{q}({\BBR}^{n-1})}\|S_{p',a}(u)\|_{L^{q}({\BBR}^{n-1})}.
\end{equation}
A $p$-elliptic matrix $A$ is also $p'$-elliptic and hence Corollary \ref{S4:C1} applies. This gives
\begin{equation}\label{S3:C7:E00oo-xx3}
\|S_{p',a}(u)\|_{L^{q}({\BBR}^{n-1})}\le C\|\tilde{N}_{p',a}(u)\|_{L^{q}({\BBR}^{n-1})}.
\end{equation}
Combining these two estimates with Proposition \ref{P3.5} we have
\begin{eqnarray}\label{S3:C7:E00oo-xx311}
\|\tilde{N}_{p,a}(u)\|^2_{L^{q}({\BBR}^{n-1})}&\lesssim&\|\tilde{N}_{2,a}(u)\|^2_{L^{q}({\BBR}^{n-1})}\\\nonumber
&\lesssim& \|S_{p,a}(u)\|_{L^{q}({\BBR}^{n-1})}\|\tilde{N}_{p',a}(u)\|_{L^{q}({\BBR}^{n-1})}\\\nonumber
&\lesssim& \|S_{p,a}(u)\|_{L^{q}({\BBR}^{n-1})}\|\tilde{N}_{p,a}(u)\|_{L^{q}({\BBR}^{n-1})}.
\end{eqnarray}
From this our claim follows.
\end{proof}

\section{Appendix: Boundary values of solutions with $\|\tilde N_{2,a}(u)\|_{L^p}<\infty$.}

The results in this section are of a general nature, and have applications to issues of  nontangential convergence of solutions in the boundary value
problems considered in this paper.

\smallskip

We start by considering an energy solution $u\in \dot{W}^{1,2}(\Omega;\mathbb C)$ of the Dirichlet boundary value problem \eqref{E:D}. Denote by 
$\tilde u:\Omega\to \mathbb C$ the averages
$$\tilde{u}(x)=\dint_{B_{\delta(x)/2}(x)} u(y)\,dy,\quad \forall x\in \Omega.$$ 
Clearly, $\tilde u$ is a continuous function on $\Omega$. We shall establish the following lemma.

\begin{lemma}\label{Traces} For $u\in \dot{W}^{1,2}(\Omega;\mathbb C)$ let $f=\text{\rm Tr }u$ be its trace on $\partial\Omega$ (which belongs to the space $\dot{B}^{2,2}_{1/2}(\partial\Omega;\mathbb C)$). 
 Then
\begin{equation}\label{Convpoint}
f(Q)=\lim_{x\to Q,\,x\in\Gamma(Q)}\tilde u(x),\qquad\text{for ${\mathcal H}^{n-1}$ a.e. }Q\in\partial\Omega.
\end{equation}
\end{lemma}
\begin{proof} It suffices to work on $\Omega={\mathbb R}^{n}_+$ since the pull-back transformation \eqref{E:rho} defines a bijection between $\dot{W}^{1,2}(\Omega;\mathbb C)$ and $\dot{W}^{1,2}({\mathbb R}^{n}_+;\mathbb C)$ and maps an interior ball $B_{\delta(x)/2}(x)$ in $\Omega$ into an open set on ${\mathbb R}^{n}_+$ that contains and is contained in balls of radius comparable to $\delta(x)/2$. Hence the result proven on 
${\mathbb R}^{n}_+$ transfers to $\Omega$.

Hence from now on let $\Omega={\mathbb R}^{n}_+$. Writing $x\in{\mathbb R}^{n}_+$ as $(x_0,x')$
consider the functions
$${\tilde u}_k(x')= {\tilde u}(2^{-k},x'),\qquad\forall x'\in{\mathbb R}^{n-1}.$$
Then for any $x,y\in {\mathbb R}^{n}_+$ with $|x-y|\le r$ and $\delta(x),\delta(y)\approx r$ we have
\begin{equation}\label{EQdiff}
|\tilde u(x)-\tilde u(y)|^2\lesssim \int _{\mathcal H}|\nabla u|^2 r^{2-n}\, dy
\end{equation}
where $\mathcal H$ is the convex hull of the set $B_{\delta(x)/2}(x)\cup B_{\delta(y)/2}(y)$.
It follows that
$$\int_{{\mathbb R}^{n-1}}|\tilde u_k(x')-\tilde u_{k+1}(x')|^2\,dx' \lesssim \int_{(2^{-(k-1)},2^{-(k+2)})\times {\mathbb R}^{n-1}}|\nabla u|^2 (2^{-k})\,dy \le 2^{-k}\|\nabla u\|^2_{L^2({\mathbb R}^{n}_+)}.$$
From this we have that $(\tilde u_k)_{k\in\mathbb N}$ is a Cauchy sequence in $L^2_{loc}({\mathbb R}^{n-1})$ and hence convergent. As $f$ is the trace of $u$ it follows that
$\tilde u_k\to f$ in $L^2_{loc}$.

Next we show that $\tilde u_k\to f$ pointwise almost everywhere. For any $\lambda>0$ consider the set
$$E_\lambda=\left\{x'\in{\mathbb R}^{n-1}:\,\forall k\in\mathbb N\text{ we have } |\tilde u_k(x')-\tilde u_{k+1}(x')|^2\le\frac{\lambda}{2^{k/2}}\right\}.$$
We estimate the size of the complement of $E_\lambda$. Clearly,
$$|E_\lambda^c|\le \sum_{k=1}^\infty \left|\left\{x'\in{\mathbb R}^{n-1}:\, |\tilde u_k(x')-\tilde u_{k+1}(x')|^2>\frac{\lambda}{2^{k/2}}\right\}\right|$$
$$\le\sum_{k=1}^\infty\frac{2^{k/2}}{\lambda}\int_{{\mathbb R}^{n-1}} |\tilde u_k(x')-\tilde u_{k+1}(x')|^2\,dx'
\le \sum_{k=1}^\infty\frac{2^{k/2}}{\lambda}2^{-k}\|\nabla u\|^2_{L^2({\mathbb R}^{n}_+)}\le \frac{C}\lambda.$$
It follows that $\cap_{\lambda>0} E_\lambda^c$ is a set of measure zero. Hence the set
$$\mathcal S=\left(\bigcup_{\lambda>0}E_\lambda\right)\cap \{x'\in {\mathbb R}^{n-1}:\, S_{10a}^1(u)(x')<\infty\}$$
has full measure. Here $S_{10a}^1(u)$ is the truncated square function at the height $1$. Clearly, $\{x'\in {\mathbb R}^{n-1}:\, S_{10a}^1(u)(x')<\infty\}$ is a set of full measure due to our assumption that $u\in W^{1,2}(\mathbb R^n_+)$. Indeed,
$$\int_{{\mathbb R}^{n-1}}S_{10a}^1(u)(x')\,dx'\approx \int_{{\mathbb R}^{n-1}\times(0,1)}|\nabla u|^2x_0\,dx_0dx'\le \|\nabla u\|^2_{L^2(\mathbb R^n_+)}<\infty,$$
and hence $S_{10a}^1(u)(x')<\infty$ a.e. as claimed and therefore $\mathcal S$ is a set of full measure.

Consider any $y\in\Gamma_a^1(x')$ for $x'\in\mathcal S$. Find an integer $k\in\mathbb N$ such that $\delta(y)\approx 2^{-k}$ and hence also $|y-(2^{-k},x'))|\approx 2^{-k}$. We estimate the difference $\tilde u(y)-\tilde u_k(x')$. As before we have (c.f. \eqref{EQdiff})
$$|\tilde u(y)-\tilde u_k(x')|^2\lesssim \int _{\mathcal H}|\nabla u(y)|^2 \delta(y)^{2-n}\, dy\lesssim \int_{{\mathcal O}_{k-2}\cup {\mathcal O}_{k-1}\cup {\mathcal O}_{k}\cup {\mathcal O}_{k+1}}|\nabla u(y)|^2 \delta(y)^{2-n}\, dy$$
where $\mathcal H$ is the convex hull of the set $B_{\delta((2^{-k},x'))/2}((2^{-k},x'))\cup B_{\delta(y)/2}(y)$.
Here 
$${\mathcal O}_j=\{(y_0,y')\in\Gamma_{10a}(x'):\, y_0\in (2^{-j},2^{-j+1}]\}.$$
Since
$$[S_{10a}^1(u)(x')]^2 = \sum_{k=1}^{\infty} \int_{{\mathcal O}_k}|\nabla u(y)|^2 \delta(y)^{2-n}\,dy <\infty$$
we see that 
$$\int_{{\mathcal O}_k}|\nabla u(y)|^2 \delta(y)^{2-n}\,dy\to 0,\qquad\text{as }k\to\infty$$
and hence 
\begin{equation}\label{eqsum}
|\tilde u(y)-\tilde u_k(x')|\to 0 \text{ as $k\to\infty$.}
\end{equation}
Consider now the sequence $(\tilde u_k(x'))_{k\in\mathbb N}$. We claim that it is Cauchy and hence convergent. Indeed, since $x'\in\mathcal S$ then $x'\in E_\lambda$ for some $\lambda>0$ and hence
$$\sum_{k=1}^\infty |\tilde u_k(x')-\tilde u_{k+1}(x')|\le\sum_{k=1}^\infty\frac{\sqrt\lambda}{2^{k/4}}<\infty.$$
From this the claim that $(\tilde u_k(x'))_{k\in\mathbb N}$ is Cauchy follows. 
As $\tilde u_k\to f$ in $L^2_{loc}$ we therefore have $\tilde u_k(x')\to f(x')$ pointwise as $k\to\infty$ for all $x'\in\mathcal S$ ($f$ can be modified on a set of measure zero if necessary). Combining this with \eqref{eqsum} we see that
$$\tilde u(y)\to f(x'),\qquad\text{as } y\to x'\text{ and }y\in\Gamma_a(x')$$
for all $x'\in\mathcal S$. This proves the lemma.
\end{proof}

\begin{lemma}\label{Traces2} Let $1<p<\infty$ and assume that the $L^p$ Dirichlet problem for the operator $\mathcal Lu=\mbox{\rm div}( A(x)\nabla u) +B(x)\cdot\nabla u$ is solvable on a domain $\Omega\subset{\mathbb R}^n$. Assume also that $\mathcal L$ is such that the Lax-Milgram lemma applies (implying existence of the energy solutions in $ \dot{W}^{1,2}(\Omega;\mathbb C)$).

For any $f\in L^p(\partial\Omega;\mathbb C)$ consider an approximation of $f$ by functions $f_k\in \dot{B}^{2,2}_{1/2}(\partial\Omega;{\BBC})\cap L^p(\partial\Omega;{\BBC})$ such that
$$f_k\to f\qquad\text{in }L^p(\partial\Omega;\mathbb C).$$

Let $u_k$ be the energy solutions corresponding to data given by $f_k$. Let
$$u=\lim_{k\to\infty} u_k\qquad\mbox{on }\Omega.$$
Then $u\in W^{1,2}_{loc}(\Omega;\mathbb C)$ and $u$ satisfies the estimate
\begin{equation}\label{Main-Est2}
\|\tilde{N}_{2,a} u\|_{L^{p}(\partial\Omega)}\leq C\|f\|_{L^{p}(\partial\Omega)}
\end{equation}
with $C>0$ as in Definition \ref{D:Dirichlet}. The averages
$$\tilde{u}(x)=\dint_{B_{\delta(x)/2}(x)} u(y)\,dy,\quad \forall x\in \Omega$$ 
satisfy
\begin{equation}\label{Convpoint2}
f(Q)=\lim_{y\to Q,\,y\in\Gamma(Q)}\tilde u(y),\qquad\text{for a.e. }Q\in \partial\Omega.
\end{equation}
\end{lemma}

We omit the proof of the lemma as it uses the same argument as in the case of real coefficients , repeatedly using the estimate that follows from solvability:
$$\|\tilde{N}_{2,a} (u_k-u_l)\|_{L^{p}(\partial\Omega)}\leq C\|f_k-f_l\|_{L^{p}(\partial\Omega)}.$$  In the real case the approximating functions are chosen so that they are continuous, which under mild assumptions on the regularity of $\partial\Omega$ (such as NTA) then immediately implies 
\begin{equation}\label{Convpoint3}
f_k(Q)=\lim_{y\to Q,\,y\in\Gamma(Q)}\tilde u_k(y),\qquad\text{for all }Q\in \partial\Omega.
\end{equation}
In our case (of complex coefficients) \eqref{Convpoint3} is replaced by \eqref{Convpoint} for each $u_k$ and $f_k$  (the a.e. convergence is sufficient for the argument). The rest of the proof goes as in the real case giving us  \eqref{Convpoint2} for $\tilde u$ and $f$.

\end{document}